\documentclass[11pt,reqno,oneside]{amsart}
\usepackage{style}

\title[Weaker cousins of Ramsey's theorem over a weak base theory]{Weaker cousins of Ramsey's theorem \\ over a weak base theory}

\author{Marta Fiori-Carones}\thanks{The authors were supported by grant no.~2017/27/B/ST1/01951 of the National Science Centre, Poland.}
\address{Institute of Mathematics, University of Warsaw, 
Banacha 2, 02-097 Warszawa, Poland}
\email{marta.fioricarones@outlook.it}
\urladdr{https://martafioricarones.github.io}

\author{Leszek Aleksander Ko\l{}odziejczyk}
\address{Institute of Mathematics, University of Warsaw, 
Banacha 2, 02-097 Warszawa, Poland}
\email{lak@mimuw.edu.pl}

\author{Katarzyna W. Kowalik}
\address{Institute of Mathematics, University of Warsaw, 
Banacha 2, 02-097 Warszawa, Poland}
\email{katarzyna.kowalik@mimuw.edu.pl}
\urladdr{}

\date{\today}

\begin{document}

\maketitle

\begin{abstract}
The paper is devoted to a reverse-mathematical study of some well-known consequences of Ramsey's theorem for pairs, focused on the chain-antichain principle $\cac$, the ascending-descending sequence principle $\ads$, and the Cohesive Ramsey Theorem for pairs $\crttt$. We study these principles over the base theory $\rcas$, which is weaker than the usual base theory $\rca$ considered in reverse mathematics in that it allows only $\Delta^0_1$-induction as opposed to $\Sigma^0_1$-induction. In $\rcas$, it may happen that an unbounded subset of $\Nb$ is not in bijective correspondence with $\Nb$. Accordingly, Ramsey-theoretic principles split into at least two variants, \lq\lq normal\rq\rq\ and \lq\lq long\rq\rq, depending on the sense in which the set witnessing the principle is required to be infinite.

We prove that the normal versions of our principles,
like that of Ramsey's theorem for pairs and two colours, are equivalent to their relativizations to proper $\Sigma^0_1$-definable cuts. Because of this, they are all $\Pi^0_3$- but not $\Pi^1_1$-conservative over $\rcas$, and, in any model of $\rcas + \neg \rca$, if they are true then they are computably true relative to some set.
The long versions exhibit one of two behaviours: they either imply $\rca$ over $\rcas$ or are $\Pi^0_3$-conservative over $\rcas$. The conservation results are obtained using a variant of the so-called grouping principle.

We also show that the cohesion principle $\coh$, a strengthening of $\crttt$, is never computably true in a model of $\rcas$ and, as a consequence, does not follow from $\rttt$ over $\rcas$.
\end{abstract}

\bigskip


\medskip




\section{Introduction}



The logical strength of Ramsey-theoretic principles has been one of the most important research topics in reverse mathematics for over two decades. Statements from Ramsey theory are an appealing subject for logical analysis,
because they are often not equivalent to any of the usual set existence principles encountered in second-order arithmetic, and they form a complex web of implications and nonimplications (see \cite{Hirschfeldt15} for an introduction to the area). Moreover, characterizing the first-order consequences of Ramsey-theoretic statements is 
frequently an interesting and demanding task.

As is the custom in reverse mathematics, 
the strength of such statements
is usually investigated over the base theory $\rca$,
a fragment of second-order arithmetic that 
includes the $\Delta^0_1$-comprehension axiom and the mathematical induction scheme for $\Sigma^0_1$-definable properties. A weaker alternative to $\rca$, introduced in \cite{SimpsonSmith} and known as $\rcas$,
allows induction only for $\Delta^0_1$
properties. Working in a weak base theory makes it possible to track nontrivial uses of induction and to make some fine distinctions that disappear over $\rca$, but it also comes with additional 
technical and conceptual challenges.

An issue of particular relevance to Ramsey-theoretic principles is that many of them assert the existence of an \emph{infinite} set $Y \subseteq \Nb$ that relates in a certain way
to a given colouring of tuples. $\rcas$ is weak enough that the precise definition of what it means to be an infinite subset of $\Nb$ becomes important. Usually, one only requires that $Y$ be unbounded in $\Nb$, and this gives rise to what we call \lq\lq normal\rq\rq\ versions of the principles. However, over $\rcas$ being unbounded is strictly weaker than being the range of a strictly increasing map with domain $\Nb$. \lq\lq Long\rq\rq\ versions of principles can be obtained by requiring $Y$ to have the latter property.

The strength of Ramsey's Theorem over $\rcas$ was investigated in \cite{yokoyama2013onthestrength} and \cite{kky:ramsey-rca0star}. The upshot of that work is that in all nontrivial cases, the normal version of Ramsey's Theorem for a fixed length of tuples and number of colours is partially conservative but not $\Pi^1_1$-conservative over $\rcas$. On the other hand, the long version of Ramsey's Theorem is strong enough to imply $\rca$.

In this paper, we ask the question whether the same general pattern also holds for other Ramsey-theoretic principles, in particular the various natural weakenings of Ramsey's Theorem for pairs that are commonly studied in reverse mathematics. Many of our results could be stated in relatively general way, but for illustrative purposes, we find it useful to concentrate on a small number of specific principles. We mostly consider two statements about linear orders, namely the chain-antichain principle $\cac$ and the ascending-descending sequence principle $\ads$, as well as the cohesive version of Ramsey's theorem for pairs and two colours $\crttt$. (The definitions are recalled in Section \ref{Subsec_normal-long}.) The statements $\cac$, $\ads$, and $\crttt$ are not only combinatorially natural, but also reasonably well-understood in the traditional reverse-mathematical setting: over $\rca$ they form a strict linear order in terms of implication, and each of them is known to be fully conservative over a  classical fragment of first-order arithmetic.

We show that normal versions of our principles,
just like those of $\rt{n}{k}$, belong to a class of statements that we call \lq\lq pseudo--second-order\rq\rq. The behaviour of any such statement in a model of $\rcas + \neg \iso$ is governed by the proper $\Sigma^0_1$-definable cuts of the model. As a consequence, normal versions 
of $\cac$, $\ads$, and $\crttt$ are $\Pi^0_3$- but not $\Pi^1_1$-conservative over $\rca$, and they have the curious feature that whenever they are true in a structure satisfying $\rcas + \neg \iso$, they are actually computably true in that structure relative to a set parameter witnessing the failure of $\iso$.
We also show that $\cac$ and $\ads$ are significantly weaker than $\rttt$ in a technical sense related to closure properties of cuts. The strength of $\crttt$ 
in this sense is left open, as is the question whether $\ads$ or $\cac$ imply $\crttt$ over $\rcas$.

We then show that long versions of Ramsey-theoretic principles tend to behave in one of two ways. Some, like $\cac$, imply $\rca$ by an easy argument dating back to \cite{yokoyama2013onthestrength}. Others,
like $\crttt$, are equivalent to normal versions of the corresponding principles in $\rcas$ or in its extension by Weak K\"onig's Lemma. As a result, these principles remain $\Pi^0_3$-conservative over $\rca$. In the case of $\ads$, both behaviours are possible depending on how exactly the principle is formalized.

We also study the cohesion principle $\coh$, a well-known strengthening of $\crttt$ that does not fit neatly into the classification into normal and long principles. It follows immediately from our results on $\crttt$ that $\coh$ is not $\Pi^1_1$-conservative over $\rcas$, which answers a question of Belanger \cite{Belanger}. Our main result about $\coh$ as such is that in contrast to many other statements we consider, it can never be computably true, even in a model of $\rcas + \neg \iso$. As a consequence, $\coh$
is not implied by $\crttt$, $\ads$, or even $\rttt$
provably in $\rcas$. 

The remainder of this paper is structured as follows. In Section 2, we discuss the necessary definitions and background, including precise formulations of the normal and long versions of our principles. We study the normal versions in Section 3, the long versions in Section 4, and $\coh$ in Section 5.

\section{Preliminaries}



We assume that the reader has some familiarity with the language of second-order arithmetic and with the most common fragments of second-order arithmetic like $\rca$ and $\wkl$,
as described in \cite{simpson:sosoa} or \cite{Hirschfeldt15}. We also assume familiarity with the usual induction and collection (or bounding) schemes encountered in first- and second-order arithmetic. Background in first-order arithmetic that is not covered in \cite{Hirschfeldt15} will be discussed below.

The symbol $\omega$ denotes the set of standard natural numbers, while $\Nb$ denotes the set of natural numbers as formalized within an arithmetic theory. In other words, if $(M,\mc{X})$ is a model of some fragment of second-order arithmetic, then $\Nb^{(M,\mc{X})}$ is simply the first-order universe $M$. The symbol $\le$ denotes the usual order on $\Nb$.

We write $\Delta^0_n$, $\Sigma^0_n$, $\Pi^0_n$ to denote the usual formula classes defined
in terms of first-order quantifier alternations, but allowing second-order free variables. On the other hand, notation without the superscript $0$, like $\Delta_n$, $\Sigma_n$, $\Pi_n$,
represents analogously defined classes of purely first-order, or ``lightface'', formulas,
that do not contain any second-order variables at all. If we want to specify the second-order parameters appearing
in a $\Sigma^0_n$ formula, we use notation like $\Sigma_n(A)$.
We extend these conventions to naming theories.
If $\Gamma$ is a class of formulas, then
$\forall\Gamma$ denotes the class of universal closures of formulas from $\Gamma$.
Note, for example, that
$\forall\Sigma^0_n$ and $\forall\Pi^0_{n+1}$
are the same class.

The theory $\rcas$, originally defined in
\cite{SimpsonSmith}, is obtained from $\rca$ by replacing the $\iso$ axiom with the weaker axiom of $\Delta^0_0$-induction (by $\Delta^0_1$-comprehension, this immediately implies induction for all $\Delta^0_1$-definable properties) and adding a $\Pi_2$ axiom $\exp$ that explicitly guarantees the totality of exponentiation. The theory $\wkls$ is obtained from $\wkl$ in an analogous way.
$\rcas$ proves the collection scheme $\bso$, and the first-order consequences of $\rcas$ and of $\wkls$ are axiomatized by $\bs{}{1} + \exp$

When we consider a model $(M,\mc{X}) \vDash \rcas$ (or work in $\rcas$ without reference to a specific model), a \emph{set} is an element of the second-order universe $\mc{X}$. 
In contrast, a \emph{$\Sigma^0_n$-definable} set or simply  \emph{$\Sigma^0_n$-set} is any subset of the first-order universe $M$ that is definable in $(M,\mc{X})$ by a $\Sigma^0_n$ formula (and likewise for $\Sigma_n$-sets, $\Pi_n$-sets etc.) A \emph{$\Delta^0_n$-definable set} or \emph{$\Delta^0_n$-set} is a $\Sigma^0_n$-set that is simultaneously a $\Pi^0_n$-set.
Since in general the models we study only satisfy $\Delta^0_1$-comprehension, 
$\Delta^0_n$-sets for $n \ge 2$ and $\Sigma^0_n$-sets for $n \ge 1$ will not always be sets. We write $\Delta_1\text{-}\mathrm{Def}(M)$ for the collection of the $\Delta_1$-definable subsets of $M$ and $\Delta^0_1\text{-}\mathrm{Def}(M,A)$ for the collection of $\Delta_1(A)$-definable subsets, where $A \subseteq M$. 
If $(M,A) \vDash \bs{}{1}(A) + \exp$,
then $(M, \Delta^0_1\text{-}\mathrm{Def}(M,A))$ is a model of $\rcas$.

Already $\id{}{0} + \exp$ is strong enough
to support a well-behaved universal $\Sigma_1$ formula $\mathrm{Sat}_1(x,y)$. We can define
the $\Sigma_1$-set $0'$ as $\{e: \mathrm{Sat}_1(e,e)\}$.

A \emph{cut} $I$ in a model of arithmetic $M$ is a downwards-closed subset of $M$ which is also closed under successor. $M$ is then an \emph{end-extension} of $I$, and it is common to write $I \subseteq_e M$, or $I \subsetneq_e M$ if $I$ is a proper cut.
The cut $I$ is a $\Sigma^0_1$-cut exactly if it is $\Sigma^0_1$-definable.

A set $A$ is \emph{unbounded} if for every $x \in \Nb$ there exists $y \in A$ with $y \ge x$. We write
$A \cof \Nb$ to indicate that $A$ is unbounded,
and more generally $A \cof B$ to indicate that $A$ is an unbounded subset of $B$. The set $A$
\emph{has cardinality $\Nb$} if it contains an $n$-element finite subset for each $n \in \Nb$,
or equivalently if it can be enumerated in increasing order as $\{a_n: n \in \Nb\}$.
Provably in $\rcas$, a set of cardinality $\Nb$
is unbounded. However, it was shown in \cite{SimpsonYokoyama}*{Lemma 3.2} that the statement \lq\lq every unbounded set has cardinality $\Nb$\rq\rq\ implies $\rca$ over $\rcas$. In other words, $\rcas + \neg \iso$
proves the existence of an unbounded set that does not have cardinality $\Nb$. Each such set can be enumerated in increasing order
as $A = \{a_i \mid i \in I\}$ for some proper $\Sigma^0_1$-cut $I$. Conversely, given a $\Sigma^0_1$ cut $I$, we can use $\Delta^0_1$-comprehension to form the set $\{\langle w_0,\ldots, w_i\rangle : i \in I\}$, where each $w_j$ is the smallest element witnessing that $j \in I$. Thus, we have:


\begin{proposition}\label{cofinalset}
Let $(M, \mathcal{X})\vDash\rcas$. For each $\Sigma^0_1$-cut $I$ there exists a set $A\in\mathcal{X}$ with $A \cof M$
that can be enumerated in increasing order as $A= \{a_i \mid i\in I\}$.
\end{proposition}

If $A= \{a_i \mid i \in I\}$ for some $\Sigma^0_1$-cut $I$, we sometimes write $a_{-1}$ for $-1$.

A bounded subset of a model $M \vDash \id{}{0} + \exp$ is \emph{coded} in $M$ if it has the form $(s)_\ack = \{x \in M \mid M \vDash x \in_{\ack} s\}$ for some $s \in M$, where $x \in_\ack s$ denotes the usual $\Delta_0$ formula expressing that the $x^{th}$ digit in the binary expansion of $s$ is $1$.
For a cut $I \subsetneq_e M$ we let $\codmi = \{ I \cap (s)_{\ack}  \mid s\in M \}$ stand for the collection of subsets of $I$ which are coded in $M$. Note that $\codmi$ can be viewed as a second-order structure on $I$. If $I$ is closed under exponentiation, then $(I, \codmi) \vDash \wkls$ \cite{SimpsonSmith}*{Theorem 4.8}.

The following lemma states an important special case of a more general result about coding in models of $\bs{0}{n} + \exp$.

\begin{lemma}[Chong-Mourad \cite{ChongMourad}] \label{ChongMourad}
Let $(M, \mc{X}) \vDash \rcas $ and let $I$ be a proper $\Sigma^0_1$-cut in $(M,\mc{X})$.
If $X \subseteq I$ is such that both $X$ and $I \setminus X$ are $\Sigma^0_1$-definable, then $X \in \codmi$.	
\end{lemma}

The iterated exponential function
is defined inductively as follows: $\exp_0(y)=1$,  $\exp_{x+1}(y)=y^{\exp_x(y)}$.
The axiom $\mathrm{supexp}$,
provable in $\rca$ but not in $\rcas$,
states that the iterated exponential function is total, i.e.~$\exp_x(y)$ exists for every $x$ and $y$.

\begin{proposition} \label{ModelExt}
For each countable $(M,\mc{X})\vDash \wkl$ there exists $K \supsetneq_e M$ such that $K \vDash \bs{}{1} + \exp$,
$M$ is a $\Sigma_1$-cut of $K$, and $\cod(K/M) = \mc{X}$. 
\end{proposition}
\begin{proof}
Suppose that $(M,\mc{X})$ is a countable model of $\wkl$. By \cite{Tanaka}, there exists
a structure $L \supsetneq_e M$ such that
$(L, \Delta_1\textrm{-}\mathrm{Def}(L)) \vDash \rca$
and $\cod(L/M) = \mc{X}$. Fix some $a \in L \setminus M$. Note that since  $L$ satisfies $\iso$ and therefore $\mathrm{supexp}$, the value $\exp_b(a)$
exists in $L$ for each $b \in L$.
Define $K \subseteq L$ so that $K = \sup (\{ \exp_m(a) \mid m\in M\})$. 
Then $K \vDash \bs{}{1} + \exp$ and $M$ is a $\Sigma_1$-cut in $K$ since $m \in M$ if and only if $K \vDash \exists y\,(y=\exp_m(a))$. Furthermore,
$\cod(K/M) = \cod(L/M) = \mc{X}$.
\end{proof}

\subsection{Normal and long versions of principles}\label{Subsec_normal-long}

Many Ramsey-theoretic statements take the form $\forall X\subseteq \Nb \, (\alpha(X) \imp \exists Y \, ({Y\textrm{ is infinite}} \land \beta(X,Y)))$, where $\alpha$ and $\beta$ are arithmetical. In this context $X$ and $Y$ are often called, respectively, \lq\lq instance\rq\rq\ and \lq\lq solution\rq\rq\ of the statement.
In $\rca$, \lq\lq $Y$ is infinite\rq\rq\ is usually formalized as \lq\lq $Y$ is unbounded\rq\rq. However, \lq\lq $Y$ is infinite\rq\rq\ could also be taken to mean
\lq\lq $Y$ has cardinality $\Nb$\rq\rq, and, as explained above, the two concepts are not equivalent in $\rcas$. Accordingly, over $\rcas$ typical Ramsey-theoretic principles will have at least two versions: one that we will take as the default and call the \emph{normal} one, in which we only require the solution $Y$ to be infinite in the sense of being unbounded; and a \emph{long} version, in which we require $Y$ to have cardinality $\Nb$.
(The word \lq\lq long\rq\rq\ is intended to emphasize that $Y$ has to be enumerated using $\Nb$ as opposed to a shorter cut.) When using standard abbreviations for various principles, we will distinguish the long versions from the normal ones by using the prefix $\ell\text{-}$.

The distinction between the two versions of Ramseyan statements was first made in the context of Ramsey's Theorem itself by Yokoyama \cite{yokoyama2013onthestrength}. 
For any $n,k \in \omega$, let $\rt{n}{k}$
be the normal version of Ramsey's Theorem for $n$-tuples and $k$ colours, \lq\lq For every
$c \colon [\Nb]^n \to k$ there exists an unbounded set $H \subseteq \Nb$ such that
$c \restric [H]^n$ is constant\rq\rq, and
let $\lrt{n}{k}$ be the long version,
which requires $H$ to have cardinality $\Nb$ (this is denoted by $\rt{n+}{k}$
in \cite{yokoyama2013onthestrength}).
It was shown in \cite{yokoyama2013onthestrength}
that $\lrttt$ implies $\iso$ over $\rcas$, while $\rcas$ extended by $\rt{n}{k}$
is $\Pi_2$-conservative over $\id{}{0} + \exp$. The study 
of $\rt{n}{k}$ over $\rcas$ was taken quite a bit further 
in \cite{kky:ramsey-rca0star}. Results obtained in that paper include the $\forall\Pi^0_3$-conservativity 
of $\rcas + \rt{n}{k}$ over $\rcas$ for each $n,k$, a complete axiomatization of $\rcas + \rt{n}{2}$ for each $n \ge 3$, and a complete axiomatization
of $\rcas + \rt{2}{2} + \neg \iso$.

The emphasis in the present paper is on principles about ordered sets, $\cac$ and $\ads$, and on the Cohesive Ramsey Theorem $\crttt$. Let us, therefore, give precise formulations of the normal and long versions for each of these principles in turn. 


The chain-antichain principle $\cac$ says that every partial order defined on $\Nb$ contains either an infinite chain or an infinite antichain. Over $\rcas$, this gives rise to the following principles.

\begin{description}[rightmargin=15mm] 
\item[$\cac$] \itshape  For every partial order $(\Nb, \preceq)$ there exists an unbounded set $S \subseteq \Nb$ which is either a chain or an antichain in $\preceq$.

\item[$\lcac$] \itshape For every partial order $(\Nb, \preceq)$ there exists a set $S \subseteq \Nb$ of cardinality $\Nb$ which is either a chain or an antichain in $\preceq$.
\end{description}

It could be argued that a more natural formulation of $\cac$ would require the existence of an unbounded chain or antichain in any partial order on an unbounded set, not necessarily on all of $\Nb$. However, we will prove in Lemma \ref{W_ADSslimInstance} that this is equivalent to the version given above and that an analogous equivalence also holds for the normal versions of other principles we study.

The ascending-descending sequence principle $\ads$ says that every linear order on $\Nb$ contains either an unbounded increasing sequence or an unbounded decreasing sequence. There is a delicate issue here, as there can be more than one way of stating the requirement that the solution to $\ads$ has to satisfy. In the literature (see e.g.~\cites{hirschfeldtShore, Hirschfeldt15}) an ascending sequence is usually taken to mean either (i) an infinite set $S \subseteq \Nb$ on which the ordering $\preceq$ agrees with the natural number ordering $\le$ or (ii) a sequence $(s_i)_{i \in \Nb}$ properly understood (that is, a map with domain $\Nb$) such that $s_0 \prec s_1 \prec s_2 \prec \ldots$ but there is no requirement on how the $s_i$ are ordered by $\le$. One could refer to these as set and sequence solutions to $\ads$, respectively. (Set and sequence solutions corresponding to descending sequences are defined analogously.) Over $\rca$, versions of $\ads$ formulated in terms of set and sequence solutions are equivalent: a set solution obviously computes a sequence solution, but given a sequence solution $(s_i)_{i \in \Nb}$ we can also obtain a set solution by taking the set of those numbers $s_j$ that are $\le$-greater than all $s_i$ for $i < j$. 

Over $\rcas$, such a thinning out argument works for the normal version of $\ads$: if we are given a sequence solution
$(s_i)_{i \in I}$ with $s_0 \prec s_1 \prec \ldots$ for some cut $I$, then the set $S$ of those $s_j$ for $j \in I$ such that $s_j > s_i$ for all $i < j$ can be obtained by $\Delta^0_1$-comprehension and is unbounded provably in $\rcas$. Thus, $S$ is a set solution to $\ads$. However, if $(s_i)_{i \in \Nb}$ is a sequence solution to the long version of $\ads$, then without $\iso$ it may happen that the set $S$ obtained in this way is no longer of cardinality $\Nb$; in other words, $S$ might not be a set solution to the long version of $\ads$. This leads us formulate the following three variants of $\ads$:  





\begin{description}[rightmargin=15mm]
\item[$\ads$] \itshape  For every linear order $(\Nb,\preceq)$  there exists an unbounded set $S \subseteq \Nb$ such that either for all $x, y\in S$ it holds that $x \le y$ iff $x \preceq y$ or for all  $x, y\in S$  it holds that $x \le y$ iff $x\succeq y$.

\item[$\ladsst$] \itshape  For every linear order $(\Nb,\preceq)$  there exists a set $S \subseteq \Nb$ of  cardinality $\Nb$ such that either for all $x, y\in S$ it holds that $x \le y$ iff $x \preceq y$ or for all  $x, y\in S$  it holds that $x \le y$ iff $x\succeq y$. 

\item[$\ladssq$] \itshape  For every linear order $(\Nb,\preceq)$ there exists a sequence $(s_i)_{i \in \Nb}$ which is either strictly $\preceq$-increasing or strictly $\preceq$-decreasing.
\end{description}

Notice that $\ladsst$ clearly implies $\ladssq$. On the other hand, it will follow from \cref{S_iso} and \cref{cor:ladssq-conservative} that the converse implication does not hold over $\rcas$. 

The final principle we focus on is the Cohesive Ramsey Theorem $\crttt$. This says that for every $2$-colouring $c$ of pairs of natural numbers, there is an infinite set $S$ on which $c$ is \emph{stable}, that is, for each $x \in S$, either $c(x,y) = 0$ for all sufficiently large $y \in S$ 
or $c(x,y) = 1$ for all sufficiently large $y \in S$. Thus, we define the following principles.  

\begin{description}[rightmargin=10mm]
\item[$\crttt$] \itshape  For every $c \colon [\Nb]^2\to 2$ there exists an unbounded set $S \subseteq \Nb$ such that for each $x \in S$ there exists $y \in S$ such that $c(x,z) = c(x,y)$ holds for all $z \in S$ with $z \ge y$.

\item[$\lcrttt$] \itshape  For every $c \colon [\Nb]^2\to 2$ there exists a set $S \subseteq \Nb$ of cardinality $\Nb$
such that for each $x \in S$ there exists $y \in S$ such that $c(x,z) = c(x,y)$ holds for all $z \in S$ with $z \ge y$.
\end{description}

We also recall some principles that are not the main focus of this work but will be mentioned in one or more contexts.

Stable Ramsey's Theorem $\srttt$ is $\rttt$ restricted to colourings $c$ that are stable on $\Nb$. 

A colouring $c \colon [A]^2 \to n$
is \emph{transitive} if $c(x,y) = c(y,z) = i$ implies $c(x,z) = i$ for all $i < n$ and all $x < y < z$ elements of $A$. The colouring $c$ is \emph{semitransitive} if the above implication holds for all $i < n$ except at most one.
The Erd\"os-Moser principle $\emo$ says that 
for any $c \colon [\Nb]^2 \to 2$, there is an infinite set $A \subseteq \Nb$ on which $c$
is transitive. 

Over $\rcas$, both $\srttt$ and $\emo$ have normal and long versions, which are defined in the natural way. $\rcas$ is able to prove the well-known equivalences of $\rttt$ with
$\srttt \land \crttt$ and with $\emo \land \ads$.

The cohesive principle $\coh$ is recalled and studied in Section \ref{sec:coh}.

\section{Normal principles} \label{Sec_Slim}

Hirschfeldt and Shore \cite{hirschfeldtShore} proved that the sequence of implications $\rt{2}{2} \imp \cac \imp \ads \imp \crttt$ holds over $\rca$. Moreover, they showed that the first and third implication do not in general reverse over $\rca$. The strictness of the implication from $\cac$ to $\ads$ was shown in \cite{Lerman2013Separating}.

It is easy to check that the proofs of the implications from $\rt{2}{2}$ to $\cac$ and $\crttt$, and of the one from $\cac$ to $\ads$, do not require $\iso$. We can thus state the following lemma (see \ref{lRT_lCAC_lADS} for its \lq\lq long\rq\rq\ counterpart).

\begin{lemma}\label{lem:easy-implications}
Over $\rcas$, the following sequences of implications hold: 
\begin{align*}
\rt{2}{2} \imp \cac \imp \ads, \\
\rt{2}{2} \imp \crttt.
\end{align*}
None of the implications can be provably reversed in $\rcas$.
\end{lemma}
Two issues left open by the lemma are whether $\ads$ or at least $\cac$ implies $\crttt$ over $\rcas$, and whether the implications above are still strict over $\rcas + \neg \iso$. 
It will be shown in \cref{separation-over-rcas} that $\rt{2}{2}$, $\cac$, $\ads$, and $\crttt$ do in fact remain pairwise distinct over $\rcas + \neg \iso$, and moreover, that they have pairwise distinct sets of arithmetical consequences. Interestingly, this is related to the fact that the principles are known to be distinct over $\wkl$. 

On the other hand, we were not able to determine whether $\rcas$ proves $\cac \imp \crttt$. This question may be related to the problem whether $\crttt$ is weaker than $\rt{2}{2}$ in a specific technical sense discussed in 
Section \ref{Subsec_closure}.

\subsection{Basic observations}\label{Subsec_instances}

In this subsection, we verify that some well-known and useful properties
of the Ramsey-theoretic principles we consider still hold
over $\rcas$. First, we show that no generality is lost
by restricting the principles to instances defined on all of $\Nb$ rather than on a more general infinite set.


\begin{lemma}\label{W_ADSslimInstance}
Over $\rcas$, each of $\rt{n}{k}$, $\cac$, $\ads$, $\crttt$ is equivalent to its generalization to orderings/colourings defined on an arbitrary unbounded subset of $\Nb$.
\end{lemma}
\begin{proof}
For $\rt{n}{k}$, this is implicit in \cite{kky:ramsey-rca0star}. The proofs are similar for all principles; we sketch them for $\ads$ and $\crttt$.

Working in $\rcas$, assume $\ads$ and let $(A,\preceq)$
be a linear order, where $A \cof \Nb$. Thus, 
$A= \{a_i \mid i \in I\}$, for some $\Sigma^0_1$-cut $I$ in $\Nb$.

Define a linear order $\preceq'$ on $\Nb$ by 
\[x \preceq' y \Biimp \exists i,j \in I \, (x \in (a_{i-1},a_{i}] \land y \in (a_{j-1},a_{j}] \land ( (i \ne j \land a_i \prec a_j) \lor (i=j \land x \le y)))
\]
That is, elements are $\preceq'$-ordered according to the the $\preceq$-ordering between the nearest elements of $A$ above them, if that makes sense, and according to the usual natural number ordering otherwise. Since $\preceq'$ is $\Delta_1(A, \preceq)$-definable, it exists as a set. Let $S' \cof \Nb$ be a strictly increasing or strictly decreasing sequence in $\preceq'$. 
Using $\Delta_1(S',A)$-comprehension, define $S \subseteq A$ by: 
\[a \in S \Biimp a \in A \land \exists x \! \le \! a \, (x \in S' \land [x,a) \cap A = \emptyset).\] It is easy to check that $S$ is unbounded and it is either a strictly increasing or a strictly decreasing sequence in $\preceq$.  

For $\crttt$, given $c \colon [A]^2 \to 2$,
use $\Delta^0_1$-comprehension to define $c' \colon [\Nb]^2 \to 2$ by:
\[c'(x, y)=
\begin{cases}
c(a_{i}, a_{j}) & \text{ if } \exists i,j \in I \, \left(i \neq j  \land x \in (a_{i-1},a_{i}] \land y \in (a_{j-1},a_{j}]\right),  \\
0 & \text{ otherwise.}
\end{cases}\] 
If $S' \cof \Nb$ is such that $c'$ is stable on $S'$, then it is easy to define analogously as above $S \cof A$
on which $c$ is stable by $\Delta_1(S',A)$-comprehension.
\end{proof}

We now check that in $\rcas$, it is still true that 
$\ads$ and $\cac$ can be viewed as the restrictions of $\rt{2}{2}$ to transitive and semitransitive colourings,
respectively. 

\begin{proposition}\label{prop:ads-trans}
Over $\rcas$, $\cac$ and $\ads$ are equivalent to $\rt{2}{2}$ restricted to semitransitive 2-colourings and to transitive 2-colourings, respectively.
\end{proposition}

\begin{proof}
This is just a verification that the arguments 
of \cite{hirschfeldtShore} go through in $\rcas$.

The implication from $\cac$ to $\rt{2}{2}$ for semitransitive 2-colourings is unproblematic. In the other direction, $\cac$
follows easily from $\rt{2}{3}$ for semitransitive 3-colourings,
which is in turn derived from $\rt{2}{2}$ for semitransitive 2-colourings. In the reduction from 3-colourings to 2-colourings, at one point we have to obtain an unbounded homogeneous set for a semitransitive 2-colouring defined on an unbounded subset of $\Nb$ rather than on $\Nb$. This is dealt with like in the proof of Lemma \ref{W_ADSslimInstance}.

The implication from $\rt{2}{2}$ for transitive 2-colourings
to $\ads$ is immediate. The other direction is \cite{hirschfeldtShore}*{Theorem 5.3}, which requires a 
comment. Given a transitive colouring $c \colon [\Nb]^2 \to 2$, we build a linear order $\preceq$ by inserting numbers $0,1,\ldots$ into it one-by-one. When $\preceq$ is already defined on $\{0,\ldots, n-1\}$, we insert $n$ into the order directly above the $\preceq$-largest $k<n$ such that
$c(k,n) = 0$; if there is no such $k$, we place $n$ at the bottom of $\preceq$. Then, we can check by induction on $n$
that the ordering $\preceq$ agrees with $c$ on $\{0,\ldots, n\}$ in the sense that for $i < j  \le  n$, we have $i \prec j$ iff $c(i,j) = 0$. In \cite{hirschfeldtShore}, $\iso$ is invoked for this purpose, but it will be clear from the above description that the induction formula is actually bounded. The induction step uses the transitivity of $c$.
\end{proof}

\subsection{Between models and cuts}\label{Subsec_cuts}
In \cite{kky:ramsey-rca0star}, it is shown that $\rt{n}{k}$ displays interesting behaviour in models of $\rcas + \neg \iso$: if $I$ is a proper $\Sigma^0_1$-cut in a model 
$(M, \mc{X})$, then $\rt{n}{k}$ holds in the entire model
$(M, \mc{X})$ if and only if it holds on the cut, that is in the structure $(I, \codmi)$. This equivalence provides important information about the first-order consequences
of $\rt{n}{k}$ over $\rcas$. It is apparent from the proof of the equivalence that it is not highly specific to $\rt{n}{k}$ and should hold for many other Ramsey-theoretic statements.

In \cref{W_Cut} below, we identify a relatively broad syntactic class of sentences that all share the property of being equivalent to their own relativizations to $\Sigma^0_1$-cuts. We then verify that Ramsey-theoretic statements such as $\rt{n}{k}$, $\cac$, $\ads$, and $\crttt$ are equivalent to sentences from that class. It follows, for instance, that all these statements
fail to be $\Pi^1_1$-conservative over $\rcas$, and that they differ in their arithmetical consequences.


\begin{definition}
The $\mathcal{L}_2$-sentence $\chi$
belongs to the class of sentences $\pso$ if there exists a sentence $\gamma$ of second-order logic in a language $(\le, R_1, \dots, R_k)$, where $k \in \omega$ and each $R_i$ is a relation symbol of arity $m_i \in \omega$, such that $\chi$ expresses:
\begin{quotation}
for any relations $R_1, \ldots,R_k$ on $\Nb$ 
and for each $D \cof \Nb$, \\*
there exists $H \cof D$ such that $(H, \le, R_1, \ldots,R_k) \vDash \gamma$.
\end{quotation}
\end{definition}
In the definition above, we slightly abuse notation by writing $(H, \le, R_1, \ldots,R_k)$ instead of the more cumbersome
$(H, {\le} \cap H^2, R_1\cap H^{m_1},\ldots, R_k\cap H^{m_k})$. The fact that this structure satisfies $\gamma$ is expressed by relativizing each first-order quantifier in $\gamma$ to $H$
and restricting each $m$-ary second-order quantifier to $m$-ary relations on $H$. Of course, when this is interpreted in a model of arithmetic $(M,\mc{X})$, \lq\lq $m$-ary relations
on $H$\rq\rq\  are understood as elements 
of $\mc{X} \cap \mc{P}(H^m)$.

The abbreviation pSO stands for \lq\lq pseudo--second-order\rq\rq:
pSO sentences appear to use both first- and second-order quantification of $\mathcal{L}_2$, 
but they are relativized to arbitrarily small unbounded subsets of $\Nb$ in such a way that in cases where $\iso$ fails their behaviour is closer to that of first-order sentences; cf.~\cref{Computably_true}.

\begin{theorem} \label{W_Cut}
If $\chi$ is a $\pso$ sentence, 
then for every $(M,\mc{X}) \vDash \rcas$ and every proper $\Sigma^0_1$-cut $I$ in $(M,\mc{X})$,
it holds that $(M,\mc{X}) \vDash \chi$ if and only if $(I,\codmi) \vDash \chi$.
\end{theorem}
\begin{proof}
Let $\gamma$ be a second-order sentence and for notational simplicity, assume that it contains only one unary relation symbol $R$ in addition to $\le$, and that all second-order quantifiers are unary. Let
$\chi$ be a pSO sentence stating that for every set $R$ and every unbounded set $D$ there exists an unbounded subset $H\cof D$ such that $(H,\le,R) \vDash \gamma$.  Let $(M, \mathcal{X})\vDash\rcas + \neg\iso$, and let  $A\in\mathcal{X}$ be a cofinal subset of $M$ 
enumerated by the cut $I$, $A=\{a_i \mid i\in I\}$, as in \cref{cofinalset}. 

Suppose first that  $(M,\mc{X}) \vDash \chi$. Let $R, D \in \codmi$ be such that $D \cof I$. 
Define $R', D' \subseteq M$ by:
\begin{align*}
x \in R' & \Biimp  \exists i \in I \,(x = a_i \land i \in R),\\ 
x \in D' & \Biimp  \exists i \in I \,(x = a_i \land i \in D).  \end{align*}
Since both $R'$ and $M \setminus R'$ are $\Sigma_1$-definable in $A$ and (the code for) $R$, we know that $R' \in \mc{X}$. Similarly, $D' \in \mc{X}$. Notice that $D'\cof M$, since $D \cof I$ and $A \cof M$. 

By our assumption that $(M,\mc{X}) \vDash \chi$, there exists $H' \in \mc{X}$  such that  $H'\cof D'$ and $(H',\le, R') \vDash \gamma$. Let $H =\{ i \in I \mid a_i \in H' \}$. 
Notice that both $H$ and $I \setminus H$ are $\Sigma_1$-definable in $H'$ and $A$,
so $H \in \codmi$ by \cref{ChongMourad}. Moreover, $H \cof D$.

To show that $(H, \le, R) \vDash \gamma$, we show that
the map $H' \ni a_i \mapsto i \in H$ induces an isomorphism
of the structures $(H', \le, R'; \mathcal{X} \cap \mc{P}(H'))$
and $(H, \le, R; \codmi \cap \mc{P}(H'))$. The fact that this map is an isomorphism between $(H', \le, R')$
and $(H, \le, R)$ follows directly from the definitions. Thus,
we only need to argue that this map also induces an isomorphism of the second-order structures $\mathcal{X} \cap \mc{P}(H')$ and $\codmi \cap \mc{P}(H')$. If $X' \in \mc{X}$
is a subset of $H'$, then $\{i \in I \mid a_i \in X'\}$
is in $\codmi$ by \cref{ChongMourad}. If $X',Y' \in \mc{X}$
are distinct subsets of $H'$, then $\{i \in I \mid a_i \in X'\}$ and $\{i \in I \mid a_i \in Y'\}$ are clearly distinct.
Finally, if $X \in \codmi$ is a subset of $H$, then
$X' = \{a_i \mid i \in X\}$ is in $\mc{X}$ by $\Delta^0_1$-comprehension, and it is a subset of $H'$.

Now suppose that $(I,\codmi) \vDash \chi$. Let $R, D \in \mc{X}$ be such that $D \cof M$.
By replacing $D$ with an appropriate unbounded subset
if necessary, we may assume w.l.o.g.~that 
$D \cap (a_{i-1},a_i]$
has at most one element for each $i \in I$.
We now transfer $R, D$ to $R',D' \subseteq I$ defined as follows:
\begin{align*}
i \in R' & \Biimp \exists x \in (a_{i-1},a_i] \cap R,\\       
i \in D' & \Biimp \exists x \in (a_{i-1},a_i] \cap D.
\end{align*}
By \cref{ChongMourad},  $R', D' \in \codmi$. 
Notice that $D' \cof I$, given that $D \cof M$.

Since $(I,\codmi) \vDash \chi$, there exists $H' \cof D'$
such that $(H',\le, R')\vDash \gamma$. Define 
\[H = \{x \in D \mid  \exists i \in H' \, (x \in (a_{i-1},a_{i}])\}.\]
Clearly $H \in \mc{X}$ and $H\cof D$.
To show that $(H,\le,R) \vDash \gamma$, it remains to prove that the structures $(H',\le,R';\codmi \cap \mc{P}(H'))$ 
and $(H,\le,R; \mc{X} \cap \mc{P}(H))$ 
are isomorphic. The isomorphism is induced by the map
that takes $i \in H'$ to the unique element of
$H \cap (a_{i-1},a_i]$. The verification that this is
indeed an isomorphism is similar to the one in the proof
of the other direction. 
\end{proof}

\begin{corollary}\label{Computably_true}
Let $\chi$ be a $\pso$ sentence 
and let $(M,\mc{X}) \vDash \rcas$. 
If $A \in \mc{X}$ is such that $(M,A) \vDash \neg \is{}{1}(A)$,
then $(M,\mc{X}) \vDash \chi$ if and only if
$(M, \Delta^0_1\text{-}\mathrm{Def}(M,A))\vDash \chi$.
\end{corollary} 
\begin{proof}
The right-hand side of the equivalence in \cref{W_Cut}
does not depend on $\mc{X}$ as long 
as a given proper cut $I$
is $\Sigma^0_1$-definable in $(M,A)$.
\end{proof}

\cref{W_Cut} and \cref{Computably_true} make it possible
to prove a very simple criterion of $\Pi^1_1$-conservativity
over $\rcas$ for $\pso$ sentences. We state the criterion 
in slightly greater generality, for boolean combinations
of $\pso$ sentences, so as to be able to conclude that some specific $\pso$ sentences have distinct sets of first-order consequences over $\rcas$.

\begin{theorem}\label{W_Pi11-cons}
Let $\psi$ be a boolean combination of $\pso$ sentences.
Then the following are equivalent:
\begin{enumerate}[(i)]
    \item $\rcas + \psi$ is $\Pi^1_1$-conservative over
    $\rcas$,
    \item $\rcas + \neg\iso \vdash \psi$,
    \item $\wkls \vdash \psi$. 
    \end{enumerate}
Moreover, if $\wkl \not \vdash \psi$,
then $\rcas + \psi$ 
is not arithmetically conservative over $\rcas$.
\end{theorem} 
\begin{proof}
The implication (iii) $\imp$ (i) is immediate from
\cite{SimpsonSmith}.

Assume that (i) holds. Note that by \cref{Computably_true}, $\rcas + \psi$ proves
the $\Pi^1_1$ statement \lq\lq for every $A$, if $\is{}{1}(A)$
fails, then $\psi$ is true in the $\Delta_1(A)$-computable
sets\rq\rq. Thus, by (i), this statement is provable in
$\rcas$. However, again by \cref{Computably_true}, in each 
model  of $\rcas + \neg \iso$ this $\Pi^1_1$ statement
is equivalent to $\psi$. This proves that (i) implies (ii).

Now assume that (iii) fails, and let $(M,\mc{X})$ be a countable model of $\wkls + \neg \psi$.
If $(M,\mc{X}) \vDash \neg \iso$, then clearly
$\rcas + \neg\iso \not \vdash \psi$. 
Otherwise, $(M,\mc{X})$ is a model of $\wkl$, so by \cref{ModelExt} there exists a structure $(K,\Delta_1\text{-}\mathrm{Def}(K)) \vDash \rcas$
in which $M$ is a proper $\Sigma^0_1$-cut and $\cod(K/M) = \mc{X}$. By \cref{W_Cut}, we get $(K,\Delta_1\text{-}\mathrm{Def}(K)) \vDash \neg \psi$.
This proves that (ii) implies (iii). 

Note also that if we do have a countable model
$(M,\mc{X})$ of $\wkl + \neg \psi$,
then the structure $(K,\Delta_1\text{-}\mathrm{Def}(K))$
constructed as in the previous paragraph
satisfies $\rcas$ but does not satisfy
the first-order statement \lq\lq $\neg \is{}{1}$ implies that 
the computable sets satisfy $\psi$\rq\rq.
This proves that if $\wkl \not \vdash \psi$, then 
$\rcas + \psi$ is not arithmetically conservative over $\rcas$.
\end{proof}

\begin{remark}
The assumption of the \lq\lq moreover\rq\rq\ part of \cref{W_Pi11-cons} could be weakened to $\wkls + \mathrm{supexp} \not \vdash \psi$, using essentially the same proof. Whether the assumption could be weakened simply to (iii) is related to the question whether every sufficiently saturated countable model of $\wkls$ is $\Sigma_1$-definable in an end-extension satisfying $\bs{}{1} + \exp$. Cf.~\cite{ky:categorical}*{Section 5}.

In \cite{fkwy}, it is shown that \emph{every} $\Pi^1_2$ sentence is $\Pi^1_1$-conservative over $\rcas + \neg \iso$ if and only if it is provable from $\wkls + \neg \iso$. Note, however, that the criterion provided by
\cref{W_Pi11-cons} applies to conservativity over $\rcas$, without $\neg \iso$ in the base theory.
\end{remark}





We now show that the general facts about $\pso$ sentences
proved above apply in particular to the Ramsey-theoretic principles we study.

\begin{lemma}\label{W_Chi}
Let $\ms{P}$ be one of the principles $\rt{n}{k}$, for $n,k \in \omega$, $\cac$, $\ads$, and $\crttt$. Then there exists 
a $\pso$ sentence $\chi$ which is provably in $\rcas$ equivalent to $\ms{P}$, both in the entire universe and on any proper $\Sigma^0_1$-cut. 
\end{lemma}
\begin{proof}
The proofs are similar for all the above principles $\ms{P}$ 
and rely on \cref{W_ADSslimInstance}. We give a somewhat detailed argument for $\ads$ and restrict ourselves to stating the appropriate $\chi$ for the other principles.

Let $\gamma$ be the sentence
\begin{quotation}
either $R$ is not a linear order \\*
or for every $x,y$ it holds that $R(x,y)$ iff $x \le y$  \\* 
or for every $x,y$ it holds that $R(x,y)$ iff $x \ge y$,
\end{quotation}
and let $\chi$ say that for every set $R$ and every unbounded set $D$, there is $H \cof D$ such that $(H, \le, R)$ satisfies $\gamma$. We claim that $\ads$ is equivalent to $\chi$ provably in $\rcas$. Clearly, if ${\preceq}$ is a linear order on $\Nb$, then $\chi$ applied with $D=\Nb$ and $R={\preceq}$ implies the existence of a set $H$ witnessing $\ads$ for $\preceq$. Thus, $\chi$ implies $\ads$. In the other direction, given a relation $R$ and an unbounded set $D$, either $R$ is a linear order on $D$ or not. In the latter case, $H=D$ witnesses $\chi$. In the former, \cref{W_ADSslimInstance} lets us apply $\ads$ to 
obtain either an ascending or a descending sequence in
$R \cap D^2$, which witnesses $\chi$. Thus, $\ads$ implies $\chi$. 

The above argument also works in a structure of the form $(I, \codmi)$ for $I$ a proper $\Sigma^0_1$-cut $I$ in a model of $\rcas$. To verify this one has to check that an analogue of \cref{W_ADSslimInstance} holds in $(I, \codmi)$,
which is unproblematic.

For $\cac$,  the corresponding $\pso$ sentence $\chi$ says that for every set $R$ and every unbounded set $D$ there exists an unbounded $H\cof D$ such that $(H, \leq, R)\vDash\gamma$, where $\gamma$ states that if $R$
is a partial order, than it is a chain or antichain.
For $\rt{n}{k}$, the sentence $\gamma$ states that if $R_1,\ldots,R_k$ form a colouring of unordered $n$-tuples, i.e.~they are disjoint $n$-ary relations whose union is the set of all $n$-tuples that are strictly increasing with respect to $\le$, then all but one of the relations $R_j$ are in fact empty. For $\crttt$, the appropriate $\gamma$ says that the binary relation $R$ is a stable colouring when restricted to the set of unordered pairs.
\end{proof}

\cref{W_Cut}, \cref{Computably_true}, and \cref{W_Chi} immediately give:

\begin{corollary} \label{rtoncuts}
Let $\ms{P}$ be one of: $\rt{n}{k}$, for each $n,k \in \omega$, $\cac$, $\ads$, and $\crttt$. Then for every $(M,\mc{X}) \vDash \rcas$ and each proper $\Sigma^0_1$-cut $I$ of $M$ it holds that $(M,\mc{X}) \vDash \ms{P}$ if and only if $(I,\codmi) \vDash \ms{P}$. If $A \in \mc{X}$ is such that $(M,A) \vDash \neg \is{}{1}(A)$, then $(M,\mc{X}) \vDash \ms{P}$ if and only if 
$(M, \Delta^0_1\text{-}\mathrm{Def}(M,A))\vDash \ms{P}$.
\end{corollary} 

For $\rt{n}{k}$, the above result was shown in \cite{kky:ramsey-rca0star}.


It follows from work of Towsner \cite{towsner:constructing-one-at-time}
that $\wkl + \cac$ does not prove $\rttt$ and $\wkl + \ads$ does not prove $\cac$. Therefore, none of the implications $\rt{2}{2} \imp \cac \imp \ads \imp \crttt \imp \top$ (where $\top$ is the constant True) available in $\rca$ can be reversed provably in $\wkl$. It thus follows from \cref{W_Pi11-cons} and \cref{W_Chi} that all principles appearing in this sequence differ in strength over $\rcas + \neg \iso$ and that they can even be distinguished by their first-order consequences over $\rcas$.

\begin{theorem}\label{separation-over-rcas}
Let $\ms{P}$ be one of the principles $\rt{2}{2}, \cac, \ads, \crttt$, and let $\ms{Q}$ be a principle to the right of $\ms{P}$ in this sequence or the constant $\top$. Then:
\begin{enumerate}[(i)]
    \item $\ms{Q}$ does not imply $\ms{P}$ over $\rcas + \neg \iso$,
    \item there is a first-order statement provable in $\rcas + \ms{P}$ but not in $\rcas + \ms{Q}$,
    \item $\rcas + \ms{P}$ is not arithmetically conservative over $\rcas$.
\end{enumerate}
\end{theorem}
\begin{proof}
By \cref{W_Chi} we can treat $\ms{P} \imp \ms{Q}$ as a boolean combination of $\pso$ sentences. Since $\wkl$ does not prove $\ms{Q}\imp\ms{P}$, \cref{W_Pi11-cons} gives (i) and additionally implies that there is an
arithmetical sentence $\theta$ provable
in $\rcas + \ms{Q}\imp\ms{P}$ but not in $\rcas$. Then $\rcas+\ms{P}\vdash\theta$ and $\rcas+\neg\ms{Q}\vdash\theta$, so $\rcas+\ms{Q}\nvdash\theta$, which proves (ii). Finally, (iii) is a special case of (ii).
\end{proof}

Together, \cref{lem:easy-implications} and \cref{separation-over-rcas} answer all questions
about provability of implications between our principles in $\rcas$ and $\rcas + \neg\iso$
except the following.

\begin{question}\label{q:cac-implies-crt}
Does $\rcas + \ads$ or $\rcas + \cac$ prove $\crttt$? 	
\end{question}


In the context of item (iii) of \cref{separation-over-rcas}, note that $\crttt$ is $\Pi^1_1$-conservative over $\rca$ \cite{CholakJockuschSlaman}, and while $\cac$ and $\ads$
are not $\Pi^1_1$-conservative over $\rca$ because they imply $\bs{0}{2}$, they are $\Pi^1_1$-conservative over $\rca + \bs{0}{2}$ \cite{CSY2012}.

We turn now to a more fine-grained analysis of conservativity issues. By \cite{kky:ramsey-rca0star}, $\rt{n}{k}$ is $\forall\Pi^0_3$-conservative over $\rcas$. \emph{A fortiori}, all the weaker principles studied in this paper are also 
$\forall\Pi^0_3$-conservative over $\rcas$. (We remark in passing
that the techniques of \cite{kky:ramsey-rca0star} show
that any $\pso$ sentence that is true in some $\omega$-model
of $\wkl$ is $\forall\Pi^0_3$-conservative over $\rcas$.)


On the other hand, if $\ms{P}$ is one of $\rt{2}{2}$,
$\cac$, and $\ads$, then the statement \lq\lq If $\is{}{1}$ fails, then any computable instance of $\ms{P}$ has a computable solution\rq\rq\ is a $\Pi_4$ sentence of first-order arithmetic.
So, essentially by \cref{rtoncuts}, we get the following nonconservation result (proved in \cite{kky:ramsey-rca0star}
for $\rt{2}{2}$).

\begin{corollary}
None of $\rt{2}{2}$, $\cac$, and $\ads$ is  $\Pi_4$-conservative over $\rcas$.
\end{corollary}

Thus, we have tight bounds on the amount of conservativity
of $\rt{2}{2}$, $\cac$, and $\ads$ over $\rcas$. 
On the other hand, the sentence \lq\lq If $\is{}{1}$ fails, then any computable instance of $\crttt$ has a computable solution\rq\rq\ is only $\Pi_5$. So, we get:

\begin{corollary}\label{cor:crttt-non-conservative}
$\crttt$ is $\forall\Pi^0_3$- but not $\Pi_5$-conservative over $\rcas$.
\end{corollary}

The following intriguing question remains open:

\begin{question}\label{W_questionCons}
Is $\wkls + \crttt$  $\forall\Pi^0_4$-conservative over $\rcas$?
\end{question}
 
\subsection{Closure properties}\label{Subsec_closure} 

To conclude our discussion of normal versions of combinatorial principles, we will show that over $\rcas$
the principle $\cac$ and all of its consequences are significantly weaker than $\rt{2}{2}$ in a technical sense related to the closure properties of cuts.

Working in $\rcas$, we write $\mathrm{I}^0_{1}$ 
to denote the definable cut consisting
of those numbers $x$ such that each unbounded set $S \cof \Nb$ contains a finite subset of cardinality $x$. 
Note that by the correspondence between unbounded subsets
of $\Nb$ and $\Sigma^0_1$-cuts stated in \cref{cofinalset}, 
$\mathrm{I}^0_1$ is simply the intersection of all $\Sigma^0_1$-cuts. 
Thus, $\mathrm{I}^0_1 = \Nb$ exactly if $\iso$ holds. 

It is easy to show in $\rcas$ that $\mathrm{I}^0_1$ is closed under multiplication: let $S$ be an infinite set that does not contain a finite set of cardinality $a^2$ and compute
a set $S'$ by taking \lq\lq every $a$-th element\rq\rq\ of $S$.
If $S'$ is finite, then some infinite end-segment 
of $S$ does not contain
any finite subset of cardinality $a$. If $S'$ is infinite,
then $S'$ itself is an infinite set with no subset of cardinality $a$, because otherwise we would find at least $a^2$ elements of $S$.

In \cite{speedup}*{Section 3}, it is shown that $\rt{2}{2}$ implies a stronger closure property, namely that $\mathrm{I}^0_1$ is closed under exponentiation. The proof of this result makes use of the well-known almost exponential
lower bounds on finite Ramsey numbers for 2-colourings of pairs. The result has some interesting consequences, among them the fact that $\rcas + \rt{2}{2}$ has nonelementary proof speedup over $\rcas$. (This was, in fact, the original motivation for
studying connections between Ramsey-theoretic principles and closure properties of  $\mathrm{I}^0_1$.) Another consequence is that
the theory $\rca + \rttt + \neg\is{}{2}$ does not prove that $\rttt$ holds in the family of $\Delta_2$-definable sets \cite{kky:ramsey-rca0star}; this rules out a potential approach to separating the arithmetical consequences of $\rca +\rttt$ from $\bs{}{2}$.

Below, we show that $\ads$ and $\cac$ are weaker than $\rt{2}{2}$ in this respect, as they do not imply the closure
of $\mathrm{I}^0_1$ under any superpolynomially growing function.
Our argument will be a typical initial segment construction,
resembling for instance the one in \cite{ky:ordinal-valued-ramsey}*{Theorem 3.3}, and it will once again make use of bounds on the finite version of the appropriate combinatorial principle. In this case, we will take advantage of the fact that, for $k \ge 2$, a partial order with $k^2-k$ elements contains either a chain or an antichain of size at least $k$. 
Indeed, by Dilworth's theorem, if the largest
antichain in a finite order has at most $k-1$ elements,
then the order can be presented as the union of $k-1$ chains.
If the order contains at least $k^2-k$ elements,
then one of those chains must have length at least $k$. 

\begin{theorem}\label{W_Inonclosure}
Let $g$ be a $\Sigma_1$-definable function such that
for every $k \in \omega$ there exists $n \in \omega$
such that $\rcas \vdash \forall x \! \ge \! n\, (g(x) \ge x^k)$. Then neither $\wkls + \cac$ nor $\wkls + \ads$ proves that $\mathrm{I}^0_1$ is closed under $g$.
\end{theorem} 
\begin{proof}
Of course, since $\cac$ implies $\ads$ over $\rcas$, it is enough to show that $\wkls + \cac$ does not imply the closure of $\mathrm{I}^0_1$ under any superpolynomially growing function $g$. 

Let $M$ be a countable nonstandard model of $\id{}{0} + \mathrm{supexp}$ and let $a \in M \setminus \omega$. We will construct a cut $I \subsetneq_e M$ in such a way that $(I, \codmi) \vDash \wkls + \cac$ and $\mathrm{I}^0_1(I,\codmi) = \sup\{a^k \mid k \in \omega\}$. 
This will suffice to prove the theorem since it will hold that $a\in\mathrm{I}^0_1(I,\codmi)$ and for any superpolynomially growing function $g$, $g(a)\notin\mathrm{I}^0_1(I,\codmi)$. 

Let $(S_n)_{n \in \omega}$ be an enumeration of all $M$-finite sets with cardinality below $a^k$ for some $k \in \omega$, let $(c_n)_{n \in \omega}$ be an enumeration of all nonstandard elements of $M$ and let $(\preceq_n)_{n \in \omega}$ be an enumeration of all $M$-coded partial orders with domain $[0, \exp_{a^a}(2)]$.

By induction on $n \in \omega$, we will construct a decreasing chain $F_0 \supseteq F_1 \supseteq F_2 \dots$ of $M$-finite sets, maintaining the condition that for each $n$ there is some $c \in M \setminus \omega$ such that $|F_n| \geq a^c$. Moreover, we will also make sure that for each $n \in \omega$, $|F_{3n}|\le a^{c_n}$, that $[\mr{min}(F_{3n+1}),\mr{max}(F_{3n+1})] \cap S_n = \emptyset$, and that $F_{3n+2}$ is either a chain or an antichain in the partial order $\preceq_n$.  

We initialize the construction by setting
$F_{-1}: = \{1,2,4,16,2^{16}, \ldots,\exp_{a^a}(2)]$. 
In step $3n$, if $|F_{3n-1}|> a^{c_n}$, let $F_{3n} \subsetneq F_{3n-1}$ be such that $|F_{3n}| = a^{c_n}$ and $\mr{min} (F_{3n}) > \mr{min} (F_{3n-1})$. Otherwise, let $F_{3n} = F_{3n-1} \setminus \{\mr{min}(F_{3n-1})\}$. 

In step $3n+1$, consider the set $S_n$. Let $k \in \omega$ be such that $|S_n|=a^k$. Assume w.l.o.g.~(by taking a proper subset of $F_{3n}$ if necessary) that $F_{3n}$ has exactly $a^c$ elements for some nonstandard $c \in M$, and let $(f_i)_{1 \le i \le a^c}$ be the increasing enumeration of $F_{3n}$. Then $F_{3n}$ can be split into $a^{k+1}$ \lq\lq intervals\rq\rq\ as follows:  
\[\{f_1,\ldots, f_{a^d}\} \cup \{f_{a^d+1}, \ldots, f_{2a^d}\} \cup \ldots \cup \{f_{(a^{k+1}-1)a^d+1},\ldots,f_{a^{k+1}a^d}\}\]
where $d=c-k-1$. Since $a^{k+1} > a^k$, the pigeonhole principle implies that there is some $i_0 < a^{k+1}$ such that $[f_{i_0a^d+1}, f_{(i_0+1)a^d}] \cap S_n =\emptyset$. 
Let $F_{3n+1}$ be the set $\{f_j \mid i_0a^d+1 \le j \le (i_0+1)a^d\}$. Notice that $|F_{3n+1}| \geq a^{c-k-1}$ and that $[\mr{min}(F_{3n+1}),\mr{max}(F_{3n+1})]  \cap S_n = \emptyset$ as wanted. 

In step $3n+2$, consider ${\preceq_n}\restric F_{3n+1}$. By construction, $|F_{3n+2}| \geq a^c$ for some nonstandard $c$. Dilworth's theorem guarantees that there exists $C \subseteq F_{3n+2}$ such that $|C| \ge a^{c/2}$ and $C$ is either a chain or an antichain in $\preceq_n$. Set $F_{3n+2} = C$.  

Finally, let $I$ be the initial segment $\sup\{\mr{min}(F_n) \mid n \in \omega\}$. We check that $I$ satisfies the requirements of our construction.

Notice that $I \subsetneq_e M$, given that $\mr{max}(F_0) \in M \setminus I$. By the construction of step $3n$, $I$ is a cut (that is, contains no greatest element) and for every $k\in\omega$, $F_k\cap I\cof I$. Moreover, $I \vDash \exp$ because $F_{-1} \cap I \cof I$,
and if $x<y$ are two elements of $F_{-1}$, then $2^x \le y$. Hence, $(I, \codmi) \vDash \wkls$ by \cite{SimpsonSmith}*{Theorem 4.8}. 

If $\preceq$ is a partial order in $\codmi$, then  ${\preceq} = {\preceq'} \cap I$ for some $M$-finite set $\preceq'$. Note that there must exist $b \in M \setminus I$ such that ${\preceq'} \restric [0,b]$ is a partial order; otherwise, $I$ would be $\Delta_0(\preceq')$-definable as the set of $i \in M$ for which ${\preceq'}\restric [0,i]$ is a poset. 
It is thus possible to extend $\preceq'$ to an order $\preceq''$ over $M$ by making every element $\le$-greater than $b$ incomparable in $\preceq''$ with all other elements of $M$. The order $\preceq''$ is $\Delta_0$-definable in $M$, so there exists $n \in \omega$ such that ${\preceq''} \restric [0,\exp_{a^a}(2)]= {\preceq_n}$. At step $3n+2$, we chose $F_{3n+2}$ as a chain or antichain in $\preceq_n$. Thus, $F_{3n+2} \cap I \in \codmi$ is  either a chain or an antichain in $\preceq$, and moreover $F_{3n+2} \cap I \cof I$. So, $(I, \codmi) \vDash \cac$.

It remains to check that $\mathrm{I}^0_1(I,\codmi)= \sup\{a^k \mid k \in \omega\}$. To prove the $\subseteq$ inclusion, let
$c \in M$ be nonstandard and let $n \in \omega$ be such that $c = c_n$. Consider $F_{3n} \cap I \in \codmi$, which is a cofinal subset of $I$. Since $F_{3n}$ has at most $a^c$ elements, then $F_{3n} \cap I$ has no finite subset of  cardinality $a^c$. To prove the reverse inclusion, we have to show that for each $k \in \omega$ and each $U \in \codmi$ such that $U \cof I$, there exists an $M$-finite set 
$V \subseteq U$ such that $|V|=a^k$. 
Let $U = U' \cap I$ for some $M$-finite set $U'$.
Note that $|U'| \ge a^k$, because otherwise
$|U'| = S_n$ for some $n \in \omega$ and the construction
of step $3n + 1$ guarantees that $[\mr{min}(F_{3n+1}),\mr{max}(F_{3n+1})] \cap U' = \emptyset$, so $U = U' \cap I \not\cof I$.
So, let $V \subseteq U'$ be the set consisting of the first $a^k$ elements of $U'$. Again, $V = S_n$ for some $n \in \omega$.
The construction of step $3n+1$ guarantees that $V' \cap I \not \cof I$. 
This means that we must have $V \subseteq U$, because otherwise there would be a largest element of $V \cap U$, then an element $u \in U \setminus V$ above it, and then an element $v \in V \setminus U$ above $u$, contradicting the definition of $V'$. Therefore, $V$ is an $M$-finite set of cardinality $a^k$ contained in $U$.  
\end{proof}

\begin{remark}
The technique used in the proof of \cref{W_Inonclosure}
can also be used to show that $\wkls + \rt{n}{k}$ does
not imply the closure of $\mathrm{I}^0_1$ under any function
of nonelementary growth rate. With a more careful choice of the initial model $M$, it can also be used to
prove slight refinements of the theorem such as the
$\forall\Pi^0_3$-conservativity of $\wkls + \cac + \lq\lq \mathrm{I}^0_1$ is not closed under any superpolynomially growing function\rq\rq\ over
$\rcas$. We do not pursue this topic further in this paper.
\end{remark}

Combining the techniques used above with the ones of 
\cite{speedup}*{Section 3}, one can show that over $\rcas$ the Erd\"{o}s-Moser principle $\emo$, or even a weakening that only requires  the solution to be an unbounded set on which a given colouring is semitransitive, implies the closure of $\mathrm{I}^0_1$ under exponentiation. This is because the lower bounds on general Ramsey numbers for pairs, along with the upper bounds on Ramsey numbers associated to orderings provided by Dilworth's theorem, imply that given $k \in \omega$, the smallest $n$ such that any 2-colouring of pairs from $\{1,\ldots,n\}$ is semitransitive on a set of size $k$ has size $2^{\Omega(\sqrt{k})}$. 

On the other hand, it is quite unclear what closure properties of $\mathrm{I}^0_1$, if any, are implied by $\crttt$. 

\begin{question}\label{q:crt-closure}
Does $\crttt$ imply that $\mathrm{I}^0_1$ is closed under $\exp$?
\end{question}

A positive answer to this question
would give negative answers to \cref{q:cac-implies-crt} and \cref{W_questionCons}. For the latter, notice that \lq\lq$\mathrm{I}^0_1$ is closed under $\exp$\rq\rq can be expressed by a $\forall\Pi^0_4$ sentence,
and \lq\lq$\mathrm{I}^0_1$ in the computable sets is closed under $\exp$\rq\rq\
can even be expressed by a purely first-order $\Pi_4$ sentence.

One reason why it is not clear whether the techniques
of \cref{W_Inonclosure} can be applied to $\crttt$ is that this principle does not have an obvious \lq\lq finite version\rq\rq\ because of the relatively high quantifier complexity of its first-order part (what is a meaningful notion of \lq\lq stable set\rq\rq\ in the finite?). Answering \cref{q:crt-closure} might require devising such a finite version of $\crttt$
(or of $\srttt$) and finding bounds on Ramsey numbers associated with it.

\section{Long principles} \label{Sec_Fat}

We now focus our attention on the long versions of Ramsey-theoretic principles.

As in the case of normal versions, many implications with an easy proof in $\rca$ transfer to $\rcas$ with no particular difficulty. Additionally, as discussed in Section \ref{Subsec_normal-long}, it is straightforward to prove that $\ladsst$ implies
$\ladssq$. The following result summarizes the \lq\lq  easy\rq\rq\  implications between our principles, as well as the non-implications known from $\rca$.

\begin{lemma}\label{lRT_lCAC_lADS}
Over $\rcas$, the following sequences of implications hold: 
\begin{align*}
\lrt{2}{2} \imp \lcac \imp \ladsst \imp \ladssq, \\
\lrt{2}{2} \imp \lcrttt.
\end{align*}
None of the implications $\lrt{2}{2} \imp \lcac \imp \ladsst$ and $\lrt{2}{2} \imp \lcrttt$ can be provably reversed in $\rcas$.
\end{lemma}

In the rest of this section, we describe some results obtained in an attempt to answer questions left open by \cref{lRT_lCAC_lADS}.
It will follow from these results (specifically from \cref{S_iso} and \cref{S_ADSGGP}) that also the implication $\ladsst \imp \ladssq$ cannot be provably reversed in $\rcas$, and that $\ladsst$ implies $\lcrttt$. 

Perhaps more interestingly, it turns out that all of the long principles we consider behave in one of two contrasting ways. Some of them are like $\lrt{2}{2}$, in that they are rather easily seen to imply $\Sigma^0_1$-induction. On the other hand, other long principles are partially conservative over $\rcas$, which makes them closer to normal principles in a well-defined technical sense. We begin by
discussing the former type of behaviour.

\begin{theorem}\label{S_iso}
Over $\rcas$, each of the principles 
$\lrttt$, $\lcac$, $\ladsst$ implies $\iso$.
\end{theorem}
\begin{proof}
The proof for $\lrttt$ was given by Yokoyama in \cite{yokoyama2013onthestrength}, and it uses a transitive colouring, so essentially the same argument works for each of the principles listed above. We describe the argument for the weakest of these principles, namely $\ladsst$.

Working in $\rcas$, suppose that $\iso$ fails,
and that an unbounded set $A$ is enumerated in increasing order as $\{a_i \mid {i\in I}\}$ for $I$ a proper $\Sigma^0_1$-cut. We define a linear order $\preceq$ on $\Nb$ in the following way:
\[
x\preceq y \Leftrightarrow 
  \begin{array}{l}
    \exists i \in I\,(x \in (a_{i-1}, a_i] \land y \in (a_{i-1}, a_i] \land x\geq y)   \\
    \lor~\exists i, j \in I\,(i < j \land x \in (a_{i-1}, a_i] \land y \in (a_{j-1}, a_j]).
  \end{array}
\]
That is, we invert the usual ordering $\le$ 
on each interval $(a_{i-1},a_i]$, but we compare elements from different intervals in the usual way. The order $\preceq$ is a set by $\Delta_1(A)$-comprehension. 

If $S\subseteq \Nb$ is such that any two elements $x, y \in S$ satisfy $x\preceq y \leftrightarrow y \leq x$, then $S$ has to be contained in an interval of the form
$(a_{i-1}, a_{i}]$, so it is finite. On the other hand, if all $x, y \in S$ satisfy $x\preceq y \leftrightarrow x \leq y$, then
$S$ can contain at most one element from each $(a_{i-1}, a_{i}]$, so the cardinality of $S$
is strictly less than $\Nb$.
\end{proof}

\begin{remark}
Note that the ordering $\preceq$
used in the proof of \cref{S_iso} is \emph{stable}, in the sense that for 
every $x$, there are only finitely many
$y$ such that $y \preceq x$. Thus,
$\iso$ is implied already by what one could
call \lq\lq $\ell$-$\sadsst$\rq\rq, the long,
set-solution version of the stable $\ads$ principle $\sads$ from \cite{hirschfeldtShore}.
\end{remark}

To show that the long versions of other principles are logically weak, we introduce an
auxiliary statement, a version of the \emph{grouping principle} $\gp^2_2$ considered
in \cite{py:rt22}. The original grouping
principle is a weakening of $\rt{2}{2}$ stating that, for any 2-colouring of pairs
and any notion of largeness of finite sets (suitably defined), there is an infinite 
sequence of large finite sets $G_0, G_1, \ldots$ (the \emph{groups}) such that for each 
$i < j$ the colouring 
is constant on $G_i \times G_j$. We consider a weaker version tailored to $\rcas$, in which the number of groups can be a proper cut, but the cardinality of individual groups should eventually exceed any finite number.

\begin{definition}
The \emph{growing grouping principle} $\ggp$ states that for every colouring 
$c\colon [\Nb]^2\to 2$ there exists a sequence of finite sets $(G_i)_{i\in I}$ such that 
\begin{enumerate}[(i)]
    \item for every $i<j \in I$ and every $x \in G_i$, $y \in G_j$ it holds that $x < y$,
    \item for every $i<j\in I$, the colouring $c\restric (G_i\times G_j)$ is constant,
    \item for every $i\in I$, $|G_i|\leq|G_{i+1}|$ and $\sup_{i\in I}|G_i|=\Nb$.
\end{enumerate}
\end{definition}

Note that $\rca + \rttt \vdash \ggp$. 
We prove a possibly surprising result on the behaviour of $\ggp$ under $\neg \iso$.  

\begin{lemma} \label{S_WKLggp}
$\wkls + \neg\iso$ implies $\ggp$. 
Moreover, $\ggp$ restricted to transitive colourings is provable in $\rcas + \neg \iso$.
\end{lemma}

\begin{remark}
\cref{S_WKLggp} implies in particular that $\ggp$ is $\Pi^1_1$-conservative over $\rcas + \neg \iso$. In contrast, Yokoyama [private communication] has pointed out that $\ggp$ is not arithmetically conservative over $\rca$. This can be seen as follows. It is shown in \cite{py:rt22}*{Theorem 5.7 \& Corollary 5.9} that $\rca$ extended by a statement $\mathsf{GP}(\mathsf{L}_\omega)$ intermediate between $\ggp$ and $\gp^2_2$ proves the principle known as 2-$\mathsf{DNC}$ and, as a consequence, an arithmetical statement $\mathrm{C}\Sigma_2$ unprovable in $\rca$. However, it is clear from the proof of \cite{py:rt22}*{Theorem 5.7} that $\rca + \ggp$
is enough for the argument to go through.
\end{remark}

\begin{proof}[Proof of \cref{S_WKLggp}]
The proof uses the technique of building a grouping by thinning out a family of finite sets
first \lq\lq from below\rq\rq\ and then \lq\lq from above\rq\rq. This method was applied to construct large finite groupings in \cite{ky:ordinal-valued-ramsey}.

Work in $\wkls + \neg\iso$, and assume that $A= \{a_i \mid i\in I\}$ is an unbounded set, where $I$ is a proper $\Sigma^0_1$-cut.
By possibly thinning out $A$ (which can only decrease $I$), we may also assume
that for each $i \in I$, $a_0>2^i$ and 
$|(a_i, a_{i+1}]|\geq a_0a_i2^{a_i}$.

Let $c \colon [\Nb]^2\to 2$. 
We want to obtain a sequence of sets $(G_i)_{i\in I}$ witnessing $\ggp$ such that $G_i\subseteq (a_{i-1}, a_{i}]$ for each $i$. 
We proceed in two main stages.

(1) We stabilize the colour \lq\lq from below\rq\rq. For each $i\in I$, build a finite sequence of finite sets 
$B^i_{-1}\supseteq B^i_0\supseteq\ldots\supseteq B^i_{a_{i-1}}$ 
in the following way. 
Let $B^i_{-1}=(a_{i-1}, a_{i}]$, and for each 
$0 \le x \le a_{i-1}$ let $B^i_{x}=\{y\in B^i_{x-1} \mid c(x, y)=k\}$, where $k\in\{0, 1\}$ is such that 
$|\{y\in B^i_{x-1}\mid c(x,y)=k\}|\geq|\{y\in B^i_{x-1}\mid c(x, y)=1-k\}|$. We can choose for instance $k = 0$ if the two values are equal.
Let $G_i'=B^i_{a_{i-1}}$. 

At this point, for each $i\in I$ and each $x \le a_{i-1}$ the colouring $c$ is constant
on $\{x\} \times G_i'$. Moreover, we have $G'_0 = [0,a_0]$ and $|G'_i| \geq a_0a_i2^{a_i-a_{i-1}-1} \ge a_0a_i$ for each $0 < i \in I$. Note that the sequence 
$(G_i')_{i\in I}$ is $\Delta_1(A)$-definable.

(2) We stabilize the colour \lq\lq from above\rq\rq.
For each $i \in I$, we can construct an infinite sequence of finite sets $G_i' = D^i_i\supseteq D^i_{i+1}\supseteq D^i_{i+2} \supseteq \ldots$
indexed by $i \le j \in I$,
with a single step of the construction
essentially like in stage (1). That is, given 
$j > i$, we let $D^i_j$ be $|\{x\in D^i_{j -1 } \mid c(x,\min G_j')=k\}|$ for that $k$ for which this set is larger. We only need to compare each $x \in D^i_{j-1}$ with one element of $G_j'$,
because we have already arranged for $c$ 
to be constant on $\{x\} \times G_j'$.
Note that for each $i < j \in I$ the colouring $c$ is constant on $D^i_j \times G_j'$. 
Also, each $D^0_j$ is nonempty,
while for $0 < i \le j \in I$ we have 
$|D^i_j| \ge a_0a_i2^{-j} \ge 2a_i$.

Intuitively, we would want to define 
$G_i = \bigcap_{i \le j \in I} D^i_j$, 
but then being a member of $G_i$ might not be
$\Delta^0_1$-definable. However, if we fix $m \in I$ and consider only the sets $\bigcap_{j=i}^m D^i_j$ for $i \le m$,
we obtain a node of length $a_m +1$ in the computable binary tree $T$ defined as follows. A finite 0-1 sequence $\tau$ belongs to $T$ if the largest $m$ such that $\mathrm{lh}(\tau) > a_m$ satisfies 
(if we identify $\tau$ with the finite set it codes): 
\begin{enumerate}[(i)]
    \item $\tau \cap [0,a_m] \subseteq\bigcup^m_{i=0}G'_i$,
    \item for every $i < j \le m$, 
    the colouring $c$ is constant on
    $(G'_i\cap\tau)\times (G'_j\cap\tau)$,
    \item $|\tau\cap G'_i|>a_i$ for every $i \le m$.
\end{enumerate}
The tree $T$ is infinite because for arbitrary $m \in I$ there exists a node in $T$ of length $a_m$, and the set $A=\{a_i \mid i\in I\}$ is unbounded. By $\mathrm{WKL}$ $T$ has an infinite path $G$ and we get the desired grouping 
$(G_i)_{i \in I}$ by taking 
$G_i=G \cap (a_{i-1}, a_{i}]$.

Now assume additionally that $c$ is a transitive colouring. By the argument from the proof of \cref{prop:ads-trans}, we can think of $c$ as given by a linear ordering $\preceq$. The first stage of the construction,
\lq\lq from below\rq\rq, is exactly as before. 
In the \lq\lq from above\rq\rq\ stage, we will make a small change. If we built the sets $D^i_j$ for $c$ as in the previous construction, then, in terms of $\preceq$, we would look at the position of $\min G_j'$ in the $\preceq$-ordering relative
to the elements of $D^i_{j-1}$.
This would split $D^i_{j-1}$ into a \lq\lq top part\rq\rq\ and a \lq\lq bottom part\rq\rq\ with respect to $\preceq$, and we would take whichever of these two parts were larger. Now, 
we will take the $\preceq$-bottom \emph{half}
of $D^i_{j-1}$ if $\min G_j'$ lies above it,
and the top $\preceq$-\emph{half} if it does not. (Do this in a way that includes the $\preceq$-midpoint in case $|D^i_{j-1}|$ is 
odd, so that $|D^i_j|$ is exactly $\lceil |D^i_{j-1}|/2 \rceil$.)

By \cref{ChongMourad}, there is an element
$s > I$ coding the set of those pairs $\langle i, j \rangle$ with $i < j \in I$ for which 
$D^i_j$ is the $\preceq$-top half of $D^i_{j-1}$. We can think of $s$ as a subset of $[0,b]\times[0,b]$
for some $b < \log a_0$. We can use $s$ to generalize the new definition of $D^i_j$ to $i \in I$ and $j \in [i,b]$: $D^i_j$ is the $\preceq$-top half of $D^i_{j-1}$ if $\langle i, j \rangle \in s$, and the $\preceq$-bottom half otherwise. Let $G_i = \bigcap_{j=i}^b D^i_j$.
It is easy to check that $(G_i)_{i \in I}$
is $\Delta^0_1$-definable and that it witnesses
$\ggp$ for $\preceq$.
\end{proof}

\begin{remark} Note that the reason
why the proof of $\ggp$ for transitive colourings does not obviously generalize to arbitrary ones is that in general,
if $i \in I < j$, then it is not clear how to split a subset of $(a_{i-1},a_i]$
into a \lq\lq more red\rq\rq\ and a \lq\lq more blue\rq\rq\ half
with respect to a (nonexistent) element of  $G_j'$. If the colouring
is transitive and given by an ordering $\preceq$, then even though we cannot actually compare the elements of $(a_{i-1},a_i]$ to
a nonexistent element, we can say which ones form the top and bottom half. 
\end{remark}

\begin{theorem}\label{S_ADSGGP}
$\rcas$ proves $\ladssq \leftrightarrow \ads$,
and $\wkls$ proves $\lcrttt \leftrightarrow \crttt$.
\end{theorem}

\begin{proof}
Let us first consider the case of $\ads$.
Clearly, $\ladssq$ implies $\ads$, and
the two principles are equivalent over $\rca$.
So, we only need to prove $\ladssq$ from $\ads$ working in $\rcas + \neg\iso$.

Let $(\Nb, \preceq)$ be an instance of $\ladssq$. By \cref{S_WKLggp}, we can apply $\ggp$ to the colouring given by $\preceq$,
obtaining a sequence of finite sets
$G_0 < G_1 < \ldots < G_i < \ldots$,
where $i\in I$ for some $\Sigma^0_1$-cut $I$. 
By \cref{W_ADSslimInstance}, we
can apply $\ads$ to the order $\preceq$ 
restricted to the set $A = \{\min(G_i)\mid i\in I\}$. Without loss of generality, assume
that this gives us an unbounded set 
$S \subseteq A$ such that for any $x, y \in S$,
$x \preceq y$ iff $x \ge y$. Assume 
$S = \{\min(G_{i_j}) \mid j \in J\}$ for some cut $J \subseteq I$. Now consider
the descending sequence in $\preceq$ defined as follows: first list the elements of $G_{i_0}$ in $\preceq$-descending order, then the elements of $G_{i_1}$ in $\preceq$-descending order, and so on. This sequence can be obtained using $\Delta_1(S, \preceq)$-comprehension,
and it has length $\Nb$, because $S \cof A \cof \Nb$, so $\sup_{j\in J}|G_{i_j}|=\sup_{i\in I}|G_i|=\Nb$. 

A similar argument shows that $\rcas + \ggp$
proves $\crttt \rightarrow \lcrttt$. However,
the instance to which we apply $\ggp$ in that argument is not necessarily transitive,
so \cref{S_WKLggp} only implies 
$\wkls + \neg \iso \vdash \crttt \imp \lcrttt$
and thus $\wkls \vdash \crttt \leftrightarrow \lcrttt$.
\end{proof}

\begin{remark}
There is version of $\ads$, called $\adsc$ in
\cite{AstorDzhafarov}, in which the solution 
is an infinite set $S$ such that either each element of $S$ has only finitely many predecessors or each element of $S$ has only finitely many successors. This principle is known to be equivalent to $\ads$ in $\rca$,
but strictly weaker in terms of Weihrauch reducibility. It is not hard to verify using
the techniques of Section \ref{Sec_Slim}
that the normal version of $\adsc$ is provably in $\rcas$ equivalent to $\ads$. Moreover,
a slight modification of the previous proof shows that the long version of $\adsc$ is also equivalent to $\ads$.
\end{remark}

\cref{S_ADSGGP} allows us to show that $\ladssq$ and $\lcrttt$ are weak principles
in the sense that they are partially conservative over $\rcas$.

\begin{corollary}\label{cor:ladssq-conservative}
Both $\wkls+\ladssq$ and $\wkls + \lcrttt$
follow from $\wkls + \rt{2}{2}$. As a consequence, these theories are $\forall\Pi^0_3$-conservative over $\rcas$ and do not imply $\iso$.
\end{corollary}
\begin{proof}
It is immediate from \cref{S_ADSGGP} and \cref{lem:easy-implications} that both 
$\wkls+\ladssq$ and $\wkls + \lcrttt$
follow from $\wkls + \rttt$.

To prove the $\forall\Pi^0_3$-conservativity
of $\wkls + \rttt$ over $\rcas$, note that
the proof of $\forall\Pi^0_3$-conservativity
of $\rcas + \rt{n}{k}$ over $\rcas$ in
\cite{kky:ramsey-rca0star} in fact shows that
any $\Sigma^0_3$ sentence consistent with
$\rcas$ is satisfied in some model of $\rcas + \rt{n}{k}$ of the form $(I, \codmi)$ for $I$ a proper $\Sigma^0_1$-cut in a model $M \vDash \ido + \exp$. By \cite{SimpsonSmith}*{Theorem 4.8}, any such model $(I, \codmi)$ satisfies $\wkls$ as well.

Of course, each theory that is at least $\Pi_1$-conservative over $\rcas$ 
is consistent with $\neg \mathrm{Con}(\id{}{0})$
and thus cannot imply even $\id{}{0} + \mathrm{supexp}$, where $\mathrm{supexp}$ expresses the totality of the iterated exponential function.
\end{proof}

Our results from Sections \ref{Sec_Slim} and \ref{Sec_Fat} on the relationships between the normal and long versions of $\rttt, \cac, \ads$, and $\crttt$ are summarized in Figure 1. One phenomenon apparent from the figure is that all of the principles considered up to this point either imply $\iso$ or are $\forall\Pi^0_3$-conservative over $\rcas$.

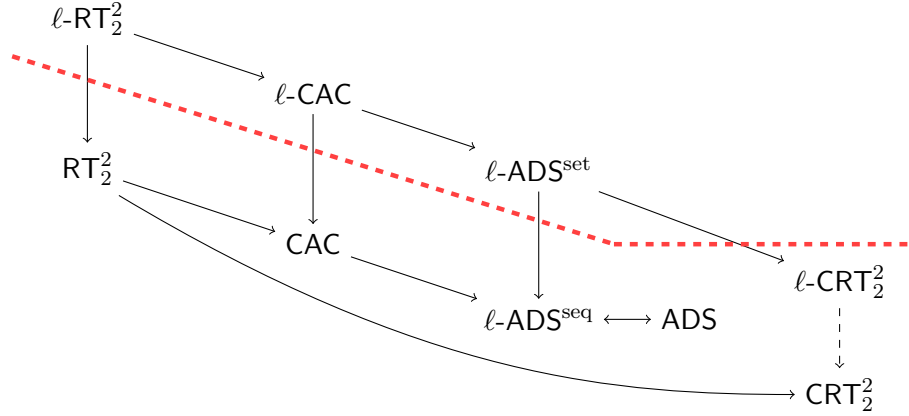
\begin{figure}

\begin{tikzpicture}
\node (lrt) at (-6,6)  {$\lrt{2}{2}$};
\node (rt) at (-6,4)  {$\rt{2}{2}$};
\node (lcac) at (-3,5)  {$\lcac$};
\node (cac) at (-3,3)  {$\cac$};
\node (ladsst) at (0,4)  {$\ladsst$};
\node (ladssq) at (0,2)  {$\ladssq$};
\node (ads) at (2,2)  {$\ads$};
\node (lcrt) at (4,2.5)  {$\lcrttt$};
\node (crt) at (4,1)  {$\crttt$};

\draw[->] (lrt) -- (rt);
\draw[->] (lrt) -- (lcac);
\draw[->] (rt) -- (cac);
\draw[->] (rt.320) .. controls (-1.5,1.2) and (1,1) .. (crt);
\draw[->] (lcac) -- (cac);
\draw[->] (cac) -- (ladssq);
\draw[->] (lcac) -- (ladsst);
\draw[->] (ladsst) -- (ladssq);
\draw[<->] (ladssq) -- (ads);
\draw[->,dashed] (lcrt) -- (crt);
\draw[->] (ladsst) -- (lcrt);
\draw[dashed,red!75,line width=0.6mm] (-7,5.5) -- (1,3);
\draw[dashed,red!75,line width=0.6mm] (1,3) -- (5,3);
\end{tikzpicture}
\caption{Summary of relations between the various versions of $\rt{2}{2}$, $\cac$, $\ads$ and $\crttt$ over $\rcas$. Solid arrows represent implications provable in $\rcas$ that do not provably reverse in $\rcas$. 
The dashed arrow represents an implication for which the reversal is open. 
Also the implications from $\cac$ and $\ads$ to $\crttt$ and from any of $\rttt, \cac, \ads$ to $\lcrttt$ are open. 
All indicated theories above the thick dashed line imply $\iso$, and all 
indicated theories below the line are $\forall\Pi^0_3$-conservative over $\rcas$.}
\end{figure}

The main open problems related to normal versions of our principles concern 
$\crttt$ and have already been stated 
in \cref{Sec_Slim}. Among the long principles,
questions about those that imply $\iso$ move
us back to the traditional realm of reverse mathematics over $\rca$. As for the weaker
long principles, an important matter is to settle the status of $\ggp$.

\begin{question}\label{Q:ggp}
Does $\rcas + \neg \iso$ imply $\ggp$? Is $\ggp$ equivalent to $\wkls$ over $\rcas + \neg \iso$?
\end{question}

A more specialized but related group of problems concerns $\lcrttt$.

\begin{question}\label{Q:lcrttt}
Is $\lcrttt$ equivalent to $\crttt$ over
$\rcas$? Does it follow from $\rcas + \rttt$?
\end{question}

By the argument used to prove \cref{S_ADSGGP}, if $\ggp$ is provable in $\rcas + \neg \iso$,
then both parts of \cref{Q:lcrttt} have a positive answer.

In the context of \cref{Q:ggp}, we mention 
a potentially interesting connection between
$\ggp$ and the long version of the Erd\"{o}s-Moser principle:
over $\rcas + \neg \iso$, $\lemo$ is equivalent to $\emo \land \ggp$. This equivalence implies in particular that $\lemo$ does not prove $\iso$ and is in fact $\forall\Pi^0_3$-conservative over $\rcas$. However, we do not know if
the $\forall\Pi^0_3$-conservative long principles considered earlier also imply $\ggp$.

To prove the equivalence, note that, on the one hand, an argument like the one in \cref{S_ADSGGP} proves $\lemo$ in $\rcas + \ggp +\emo$. Given $c\colon [\Nb]^2 \to 2$,
we can use $\ggp$ to obtain $(G_i)_{i \in I}$ such that $c\restric (G_i\times G_j)$ is constant for each $i < j \in I$, thin out each $G_i$ at most exponentially to obtain $G_i'$ on which $c$ is constant, and then apply $\emo$ to 
$c \restric \{\min(G_{i}'\mid i \in I\}$
in order to find $S = \{\min(G_{i_j}'\mid j \in J\}$ on which $c$ is transitive. Then $\bigcup_{j \in J} G_{i_j}'$ is a set of cardinality $\Nb$ on which $c$ is transitive.  On the other hand,
$\rcas + \neg\iso + \lemo$ implies $\ggp$. Given a colouring $c\colon [\Nb]^2\to 2$, we can apply $\lemo$ to obtain a set $S$ of cardinality $\Nb$ such that $c$ is transitive on $S$. Then \cref{S_WKLggp} applied to
$c\restric [S]^2$ provides a solution to $\ggp$.






\section{The curious case of $\coh$}\label{sec:coh}

In the final section of the paper, we consider the behaviour over $\rcas$ of the cohesion
principle $\coh$. Recall that a set  $C \subseteq \Nb$ is \emph{cohesive} for a sequence $(R_n)_{n \in \Nb}$ of subsets
of $\Nb$ if, for each $n \in N$, either all but finitely many elements of $C$ belong to $R_n$
or all but finitely many elements of $C$ belong to $\Nb \setminus R_n$. 
We write $C \subseteq^* R_n$ in the former case
and $C \subseteq^* \overline{R_n}$ in the latter.

\begin{description}[rightmargin=10mm]
\item[$\coh$]\itshape For each sequence $(R_n)_{n \in \Nb}$ of subsets of $\Nb$, there exists an unbounded set $C$ which is cohesive for $(R_n)_{n \in \Nb}$.
\item[$\lcoh$]\itshape For each sequence $(R_n)_{n \in \Nb}$ of subsets of $\Nb$, there exists a set $C$ of cardinality $\Nb$ which is cohesive for $(R_n)_{n \in \Nb}$.
\end{description}

Belanger \cite{Belanger} asked whether
$\coh$ is $\Pi^1_1$-conservative over
$\rcas$. A negative answer to this question
follows from the results of Section \ref{Sec_Slim}. This is because $\coh$ implies
$\crttt$, and the implication remains provable in $\rcas$: for a colouring $c \colon [\Nb]^2 \to 2$, any set $C$ that is cohesive 
for the sequence $(\{y \mid c(n,y)=1\})_{n \in \Nb}$ is also stable for $c$. Thus, \cref{cor:crttt-non-conservative} immediately implies the following result.

\begin{corollary}
$\rcas + \coh$ is not $\Pi_5$-conservative over
$\rcas$.
\end{corollary}


Of course, $\lcoh$ implies $\lcrttt$ over $\rcas$ in an analogous way. Below we focus on $\coh$, as we have no results to report on $\lcoh$ beyond immediate consequences of the easy implications from $\lcoh$ to $\coh$ and to $\lcrttt$. 

In terms of our classification of Ramsey-theoretic statements
into normal and long principles, $\coh$ has some aspects of both. On the one hand, the solution $C$ is only required to be unbounded but not to have cardinality $\Nb$. On the other hand, $C$ is required to behave in a certain way with respect to \emph{each} element of the sequence $(R_n)_{n \in \Nb}$, which obviously has length $\Nb$. We will show that the latter
feature of $\coh$ has an interesting consequence: the well-known implication from $\rt{2}{2}$ to $\coh$ \cite{CholakJockuschSlaman}\footnote{The proof of $\rt{2}{2} \imp \coh$ given in \cite{CholakJockuschSlaman} actually requires $\ist$ but Mileti \cite{MiletiPhD} gave another proof which goes through in $\rca$.}
is not provable over $\rcas$. \emph{A fortiori}, this means that neither the implications from $\cac$ and $\ads$ to $\coh$ known to hold over $\rca$ nor the equivalence between $\coh$ and $\crttt$ known to hold over $\rca + \bs{0}{2}$ \cite{hirschfeldtShore} are provable in $\rcas$.





To prove that the implication
$\rttt \imp \coh$ breaks down over
$\rcas$, we will show that, in contrast
to all the ``normal'' Ramsey-theoretic principles considered in Section \ref{Sec_Slim}, $\coh$ is never computably 
true, i.e.~it never holds in a model
of the form $(M,\Delta_1\text{-}\mathrm{Def}(M))$. We will prove this by means of a detour through what is called the $\Sigma^0_2$-separation principle in \cite{Belanger}.

\begin{description}[labelwidth=25mm,rightmargin=15mm]
\item[$\Sigma^0_2$-$\mathrm{separation}$] \itshape For every two disjoint $\Sigma^0_2$-sets $A_0$, $A_1$ \\ there exists a $\Delta^0_2$-set $B$ such that $A_0\subseteq B$ and $A_1\subseteq\overline{B}$. 
\end{description}

It was shown in \cite{Belanger} that $\coh$ is equivalent to  $\Sigma^0_2$-separation
over $\rca + \bs{0}{2}$ and that the implication from $\coh$ to $\Sigma^0_2$-separation works over $\rca$.
Below, we verify that this implication
remains valid over $\rcas$.
On the other hand, we show that $\mathrm{B}\Sigma_1+\exp$ is enough to 
prove the existence of two disjoint lightface $\Sigma_2$-sets that cannot be separated by a $\Delta_2$-set.
That is the same thing as saying that in any structure $M \vDash \bs{}{1}+\exp$, the second-orded universe consisting exclusively of the $\Delta_1$-definable sets satisfies the negation of the $\Sigma^0_2$-separation principle and hence also $\neg \coh$. 


\begin{lemma}\label{lem:coh-sep}
$\rcas$ proves that $\coh$ implies $\Sigma^0_2$-$\mathrm{separation}$.
\end{lemma} 
\begin{proof}
We will follow the structure of the proof in $\rca$ described in \cite{Belanger} (which is based on \cite{jockusch-stephan}), pointing out where we have to depart from it. We work in $\rcas + \coh$ and prove the dual formulation of $\Sigma^0_2$-separation: if $A_0$ and $A_1$ are $\Pi^0_2$ sets such that $A_0 \cup A_1 = \Nb$, then there exists a $\Delta^0_2$-set $B$ such that $B \subseteq A_0$ and $\overline B \subseteq A_1.$

Assume that:
\begin{align*}
A_0 = & \{x \mid \forall y\, \exists z\, \theta_0(x, y, z)\}, \\    
A_1 = & \{x \mid \forall y\, \exists z\, \theta_1(x, y, z)\},
\end{align*}
where $\theta_0$, $\theta_1$ are $\Delta^0_0$, and for each $n \in \Nb$ it holds that $n\in A_0$ or $n\in A_1$. 

The argument in $\rca$ would now make use of a $\Delta^0_1$-definable function 
$f \colon  \Nb \times \Nb \to 2$ 
such that for every $n$, 
\[
\{s \mid f(n, s)=i\} \textrm{ is infinite iff } n \in A_i.
\]
It seems unclear whether we can have access to such a function in $\rcas$. However, we can use a witness comparison argument to find a $\Delta^0_1$-definable  $f \colon  \Nb \times \Nb \to 2$ 
such that for every $n$, 
\[
\textrm{if } \{s \mid f(n, s)=i\} \textrm{ is infinite, then } n\in A_i.
\]
Namely, for every $n$ at least one of $\forall y\, \exists z\, \theta_0(n, y, z)$ and $\forall y\, \exists z\, \theta_1(n, y, z)$ holds.
So, by $\bso$, for every $n$ and $s$ there must exist some $w_0$ such that 
$\forall y\!\le\!s\, \exists z\!\le\!w_0\, \theta_0(n, y, z)$
or some $w_1$ such that 
$\forall y\!\le\!s \, \exists z\!\le\!w_1\, \theta_1(n, y, z)$. Define $f(n, s) = 0$ 
if the smallest such $w_0$ is at most equal to the smallest such $w_1$, and $f(n,s) = 1$ otherwise.

Now consider the $\Delta^0_1$-definable  sequence of sets $(R_n)_{n\in \Nb}$ where $R_n=\{s \mid f(n, s) = 0\}$. Let $C$ be a cohesive set for this sequence. Notice that if $C\subseteq^*R_n$, then $R_n$ is infinite and hence $n \in A_0$, and analogously if $C\subseteq^*\overline{R}_n$ then $n \in A_1$. 

Let
\[
B=\{n \mid \exists k\,\forall \ell \!\ge\! k \,(\ell \in C \imp \ell \in R_n)\}.
\]
Since $C$ is cohesive for $(R_n)_{n \in \Nb}$, both $B$ and its complement are $\Sigma^0_2$-definable. 
Moreover, it follows from the construction that if $n \in B$ then $n \in A_0$ and if $n \notin B$ then $n \in A_1$.
\end{proof}

\begin{lemma}\label{delta-2-nonsep}
$\bs{}{1}+\exp$ proves that there exist two disjoint $\Sigma_2$-sets that cannot be separated by a $\Delta_2$-set.
\end{lemma}

\begin{proof}
We verify that an essentially standard proof of the existence of $\Delta_2$-inseparable disjoint $\Sigma_2$-sets goes through in $\bs{}{1} + \exp$. The recursion-theoretic facts and notions needed for the proof to work were formalized within $\bs{}{1} + \exp$ in \cite{chong-yang:jump-cut}.

A \emph{Turing functional} $\Phi$ is a $\Sigma_1$-set of tuples $\langle x,y, P, N \rangle$, where $x,y \in \Nb$ and $P, N$ are disjoint finite sets. Turing functionals are constrained to be well-defined in the sense that for fixed $x, P, N$ there is at most one $y$ such that $\langle x,y, P, N \rangle\in\Phi$, and to be monotone in the sense that increasing $P$ or $N$ preserves membership in $\Phi$. Given a Turing functional $\Phi$, we say that $\Phi^{0'}(x) = y$ if there exist $P \subseteq 0'$ and $N \subseteq \overline{0'}$ such that $\langle x,y, P, N \rangle \in \Phi$. 

Work in $\bs{}{1} + \exp$, and let $(\Phi_e)_{e \in \Nb}$ be an effective listing of all Turing functionals. Let $A_0$ be the $\Sigma_2$-set $\{e \in \Nb: \Phi^{0'}_e(e) = 0 \}$, and let $A_1$ be the $\Sigma_2$-set $\{e \in \Nb: \Phi^{0'}_e(e) = 1\}$. Clearly, $A_0$ and $A_1$ are disjoint. We claim that they cannot be separated by a $\Delta_2$-set.

Suppose that $B$ is a $\Delta_2$-set such that
$A_0 \subseteq B$ and $A_1 \subseteq \overline B$. By \cite{chong-yang:jump-cut}*{Corollary 3.1}, provably in $\bs{}{1} + \exp$ the $\Delta_2$-set $B$ is \emph{weakly recursive} in $0'$ in the following sense: there is some Turing functional $\Phi_{e_0}$ such that for every $x$, if $x \in B$ then $\Phi^{0'}_{e_0}(x) = 1$, and if $x \notin B$ then $\Phi^{0'}_{e_0}(x) = 0$. By the definition of $A_0$ and $A_1$, this implies that $\Phi^{0'}_{e_0}(e_0) = 0$ iff 
$\Phi^{0'}_{e_0}(e_0) = 1$, which is a contradiction because $\Phi^{0'}_{e_0}$ is defined on every input and takes 0/1 values.
\end{proof}

\begin{theorem}\label{coh-computably-false}
Any model $(M,\Delta^0_1\text{-}\mathrm{Def}(M,A)) \vDash \rcas$, where $A \subseteq M$, satisfies $\neg \coh$. 
\end{theorem}
\begin{proof}
This is an immediate consequence of \cref{lem:coh-sep} and \cref{delta-2-nonsep} relativized to $A$. \cref{lem:coh-sep} says that if the structure $(M,\Delta^0_1\text{-}\mathrm{Def}(M,A))$ satisfied $\coh$, then it would also satisfy the $\Sigma^0_2$-separation principle. 
The latter would contradict \cref{delta-2-nonsep}, because in
$(M,\Delta^0_1\text{-}\mathrm{Def}(M,A))$ the $\Sigma^0_2$-sets are exactly the $\Sigma_2(A)$-definable sets and the
$\Delta^0_2$-sets are exactly the $\Delta_2(A)$-definable sets.
\end{proof}

\begin{corollary}\label{rtttdoesnotimplycoh}
$\rcas+\rttt$ does not imply $\coh$.
\end{corollary}

\begin{proof}
By \cref{coh-computably-false}, it is enough to note that there exists a model of
$\rcas + \rttt$ of the form
$(M,\Delta^0_1\text{-}\mathrm{Def}(M,A))$ for some $A \subseteq M$. The existence of such a model follows from the existence of a model of $\rcas + \rttt + \neg \iso$ \cite{kky:ramsey-rca0star} 
and \cref{Computably_true}.
\end{proof}

\begin{corollary}
$\rt{2}{2}$, $\cac$, and $\ads$ are incomparable with $\coh$ with respect to implications over $\rcas$. 	
\end{corollary}


Another consequence of \cref{coh-computably-false}
is that an analogue of \cref{W_Cut} does not hold for $\coh$. In particular, it is not true that if $(M, \mc{X})\vDash \rcas$ and $(I, \codmi)\vDash \coh$ for some $\Sigma^0_1$-cut $I$ of $M$, then $(M, \mc{X})\vDash \coh$, since in a model of $\neg \is{}{1}$ this would work in particular for $\mc{X} = \Delta_1\text{-Def}(M)$. On the other hand, using methods in the style of Section \ref{Sec_Slim} it is easy to show that the converse implication still holds.

\begin{proposition}
For every $(M,\mc{X}) \vDash \rcas$ and every proper $\Sigma^0_1$-cut $I$ in $(M,\mc{X})$, if $(M,\mc{X}) \vDash \coh$, then $(I,\codmi) \vDash \coh$.
\end{proposition}
\begin{proof}
Suppose $(M, \mathcal{X})\vDash\rcas+\coh$ and $I$ is a proper $\Sigma^0_1$-cut in $(M,\mc{X})$. Let  $A\in\mathcal{X}$ be a cofinal subset of $M$ enumerated as $A=\{a_i \mid i\in I\}$, as in \cref{cofinalset}. 

Let $(R_i)_{i \in I}$ be a sequence of subsets of $I$ that belongs to $\codmi$. Define a sequence $(R'_n)_{n \in M}$ in the following way. If $n \in (a_{i-1},a_i]$ for some $i \in I$, let \[R'_n=\{x \in M \mid \exists j \in I\,(x \in (a_j,a_{j+1}]) \land j \in R_i)\}.\] The sequence $(R'_n)_{n \in M}$ is $\Delta_1$-definable in $A$ and the code for $(R_i)_{i \in I}$, so it belongs to $\mc{X}$.
By $\coh$ in $(M,\mc{X})$, there exists $C' \in \mc{X}$ such that $C' \cof M$ and $C'$ is cohesive for $(R'_n)_{n \in M}$. Define $C = \{i \in I \mid C' \cap (a_i,a_{i+1}] \ne \emptyset \}$. Both $C$ and $I \setminus C$ are $\Sigma_1$-definable in $C'$ and $A$, so $C \in \codmi$ by \cref{ChongMourad}. Moreover, $C \cof I$ and it is easy to check that $C$ is cohesive for $(R_i)_{i \in I}$.
\end{proof}

Results such as \cref{coh-computably-false} and \cref{rtttdoesnotimplycoh} provide some new information about $\coh$, but the strength of this principle in $\rcas$ is still to a large extent mysterious. Some rather basic problems remain open.

\begin{question}
Does $\coh$, or at least $\lcoh$, imply $\iso$ over $\rcas$? Is $\lcoh$, or at least $\coh$,
$\forall\Pi^0_3$-conservative over $\rcas$?
\end{question}

\begin{question}
Does $\rcas$, or at least $\wkls$, prove $\coh \biimp \lcoh$?
\end{question}
\medskip

\subsection*{Acknowledgement}
The authors are grateful to Tin Lok Wong 
and Keita Yokoyama for valuable comments on
an early draft version of this work.

\bibliographystyle{alpha}
\bibliography{bibliography}

\end{document}